\documentclass[11pt,a4paper,reqno]{amsart}
\usepackage{amsfonts,amsthm,amscd,epsfig,amsmath,amssymb,enumerate}

\author[P. Caputo]{Pietro Caputo}
\address{Pietro Caputo. Dip. Matematica, Universita' di Roma Tre, L.go S. Murialdo 1,
00146 Roma, Italy} \email{caputo\@@mat.uniroma3.it}
\author[A. Faggionato]{Alessandra Faggionato} \address{Alessandra Faggionato. Dip. Matematica,
Universit\`a di Roma ``La Sapienza'', P.le Aldo Moro 2, 00185 Roma,
Italy} \email{faggiona\@@mat.uniroma1.it}
\author[A.  Gaudilli\`{e}re]{Alexandre Gaudilli\`ere} \address{Alexander Gaudilli\`ere. Dip. Matematica, Universita' di Roma Tre, L.go S.
Murialdo 1, 00146 Roma, Italy} \email{gaudilli\@@mat.uniroma3.it}

\title[random walk on a
  random point process]
{Recurrence and transience  for long--range reversible
random walks on a
         random point process}

\numberwithin{equation}{section}

\DeclareMathSymbol{\leqslant}{\mathalpha}{AMSa}{"36} 
\DeclareMathSymbol{\geqslant}{\mathalpha}{AMSa}{"3E} 
\DeclareMathSymbol{\eset}{\mathalpha}{AMSb}{"3F}     
\renewcommand{\leq}{\;\leqslant\;}                   
\renewcommand{\geq}{\;\geqslant\;}                   

\newcommand{\la}{\label}

\newcommand{\be}{\begin{equation}}

\def\1{\ifmmode {1\hskip -3pt \rm{I}} \else {\hbox {$1\hskip -3pt \rm{I}$}}\fi}

\newtheorem{Th}{Theorem}[section]
\newtheorem{Le}[Th]{Lemma}
\newtheorem{Pro}[Th]{Proposition}
\newtheorem{Cor}[Th]{Corollary}

\newcommand{\cC}{\ensuremath{\mathcal C}}
\newcommand{\cD}{\ensuremath{\mathcal D}}
\newcommand{\cE}{\ensuremath{\mathcal E}}
\newcommand{\cF}{\ensuremath{\mathcal F}}

\newcommand{\cK}{\ensuremath{\mathcal K}}

\newcommand{\cN}{\ensuremath{\mathcal N}}

\newcommand{\cP}{\ensuremath{\mathcal P}}

\newcommand{\bbC}{{\ensuremath{\mathbb C}} }

\newcommand{\bbE}{{\ensuremath{\mathbb E}} }

\newcommand{\bbN}{{\ensuremath{\mathbb N}} }

\newcommand{\bbP}{{\ensuremath{\mathbb P}} }

\newcommand{\bbR}{{\ensuremath{\mathbb R}} }

\newcommand{\bbZ}{{\ensuremath{\mathbb Z}} }

\newcommand{\si}{\sigma}

\newcommand{\wt}{\widetilde}

\let\a=\alpha \let\b=\beta   \let\d=\delta  \let\e=\varepsilon
 \let\g=\gamma     \let\k=\kappa  \let\l=\lambda
      \let\o=\omega    \let\p=\pi  
\let\r=\rho  \let\s=\sigma    
\let\y=\upsilon \let\x=\xi 
\let\D=\Delta   \let\G=\Gamma  \let\L=\Lambda 
\let\O=\Omega      

\def\\{\hfill\break}

\def\thsp{\thinspace}
\def\x{\thinspace}
\def\tthsp{\kern .083333 em}

\def\?{\mskip -10mu}
\def\indbox#1{\hbox to \parindent{\hfil\ #1\hfil} }

\newfam\msafam
\newfam\msbfam
\newfam\eufmfam
\def\hexnumber#1{%
  \ifcase#1 0\or 1\or 2\or 3\or 4\or 5\or 6\or 7\or 8\or
  9\or A\or B\or C\or D\or E\or F\fi}
\font\tenmsa=msam10
\font\sevenmsa=msam7
\font\fivemsa=msam5
\textfont\msafam=\tenmsa
\scriptfont\msafam=\sevenmsa
\scriptscriptfont\msafam=\fivemsa
\edef\msafamhexnumber{\hexnumber\msafam}%
\mathchardef\restriction"1\msafamhexnumber16
\mathchardef\ssim"0218
\mathchardef\square"0\msafamhexnumber03
\mathchardef\eqd"3\msafamhexnumber2C
\def\QED{\ifhmode\unskip\nobreak\fi\quad
  \ifmmode\square\else$\square$\fi}
\font\tenmsb=msbm10
\font\sevenmsb=msbm7
\font\fivemsb=msbm5
\textfont\msbfam=\tenmsb
\scriptfont\msbfam=\sevenmsb
\scriptscriptfont\msbfam=\fivemsb

\font\teneufm=eufm10
\font\seveneufm=eufm7
\font\fiveeufm=eufm5
\textfont\eufmfam=\teneufm
\scriptfont\eufmfam=\seveneufm
\scriptscriptfont\eufmfam=\fiveeufm

\def\({\left(}
\def\){\right)}
\let\neper=e
\let\ii=i

\def\nep#1{ \neper^{#1}}

\def\tc{\thsp | \thsp}
\let\<=\langle
\let\>=\rangle

\def\tr{ \mathop{\rm tr}\nolimits }

\outer\def\nproclaim#1 [#2]#3. #4\par{\medbreak \noindent
   \talato(#2){\bf #1 \Thm[#2]#3.\enspace }%
   {\sl #4\par }\ifdim \lastskip <\medskipamount
   \removelastskip \penalty 55\medskip \fi}
\def\thmm[#1]{#1}
\def\teo[#1]{#1}
\def\sttilde#1{%
\dimen2=\fontdimen5\textfont0
\setbox0=\hbox{$\mathchar"7E$}
\setbox1=\hbox{$\scriptstyle #1$}
\dimen0=\wd0
\dimen1=\wd1
\advance\dimen1 by -\dimen0
\divide\dimen1 by 2
\vbox{\offinterlineskip%
   \moveright\dimen1 \box0 \kern - \dimen2\box1}
}
\newcommand{\varphipa}{\varphi_{{\rm p},\alpha}}
\newcommand{\varphieb}{\varphi_{{\rm e},\beta}}

\begin{document}

\maketitle

\begin{abstract}
We consider reversible random walks in random environment obtained from symmetric
long--range jump rates on a random point process. We prove almost sure transience
and recurrence results under suitable assumptions
on the point process and the jump rate function.
For recurrent models we obtain
almost sure estimates on effective resistances in finite boxes.
For transient models we construct explicit fluxes with finite energy on the
associated electrical network.

\bigskip


\smallskip

\noindent {\em MSC-class:} 60K37; 60G55; 60J45

\end{abstract}

\thispagestyle{empty}
\section{Introduction and results}
We consider random walks in random environment obtained as random
perturbations of long--range random walks in deterministic
environment.
Namely, let $S$ be a 
locally finite subset of $\bbR^d$, $d\geq 1$ and call $X_n$ the
discrete time Markov chain with state space $S$ which jumps from a
site  $x$ to  another  site $y$ with probability $p(x,y)$
proportional to $\varphi(|x-y|)$, where
$\varphi:(0,\infty) \to (0,1]$ is a positive bounded measurable function and $|x|$
stands for the Euclidean norm of $x\in\bbR^d$.
We write $P$ for the law of $X_n$, so that for $x\neq y\in S$:
$$
P(X_{n+1}=y\tc X_n=x) =
p(x,y):=\frac{\varphi(|y-x|)}{w_S(x)}
\,,
$$
where we define $w_S(x):=\sum_{z\in S:\,z\neq x} \varphi(|z-x|)$.
 Note
that the random walk $X_n$ is well defined as soon as
$w_S(x)\in(0,\infty)$ for every $x\in S$.   In this case,
$w_S=\{w_S(x)\,,\;x\in S\}$ is a reversible measure, i.e.\
$w_S(x)p(x,y)$ is symmetric. Since the random walk is irreducible
due to the strict positivity of $\varphi$,  $w_S$ is the unique
invariant measure up to a multiplicative
constant. We shall often speak of
the random walk  $(S,\varphi)$ when we need to emphasize the
dependence on the state space $S$ and the function $\varphi$.
Typical special cases of functions $\varphi$ will be the
polynomially decaying   function  $\varphipa(t) :=1\wedge
t^{-d-\a}$, $\a>0$ and the stretched exponential function
$\varphieb(t) :=\exp (-t^{\b})$, $\b>0$.

We investigate here the transience and recurrence of the
random walk $X_n$.
We recall that   $X_n$  is said to be {\em recurrent} if for some $x\in S$,
the walk started at $X_0=x$ returns to $x$ infinitely many times
with probability one. Because of irreducibility if this happens at
some $x\in S$ then it must happen at all $x\in S$. $X_n$ is said to
be {\em transient} if it is not recurrent.
If we fix $S=\bbZ^d$, we
obtain standard homogeneous lattice walks. Transience and recurrence
properties of these walks can be obtained by classical harmonic
analysis, as extensively discussed e.g.\ in
Spitzer's book \cite{Spitzer}  (see also Appendix \ref{losangeles}).
For instance, it is well known that for dimension $d\geq 3$ both
$(\bbZ^d,\varphieb)$ and $(\bbZ^d,\varphipa)$ are transient for all
$\b>0$ and $\a>0$ while for $d=1,2$, $(\bbZ^d,\varphieb)$  is
recurrent for all $\b>0$ and $(\bbZ^d,\varphipa)$ is transient iff
$0<\a<d$.

We shall be interested in the case where $S$ is a locally finite
{\em random} subset of $\bbR^d$, i.e.\ the realization of a simple point
process on $\bbR^d$. We denote by $\bbP$ the law of the point
process. For this model to be well defined for $\bbP$--almost all
$S$ we shall require that, given the choice of $\varphi$:
\be\la{pw}
\bbP\left( w_S(x) \in(0,\infty)\,,\;\;{\rm for \;all\;}\,x\in S\right)=1\,.
\end{equation}

If we look at the set $S$ as a random perturbation of the regular lattice $\bbZ^d$,
the first natural question is to find conditions on the law of the point process $\bbP$
and the function $\varphi$ such that $(S,\varphi)$ is $\bbP$--a.s. transient (recurrent)
iff $(\bbZ^d,\varphi)$ is transient (recurrent). In this case we say that the random walks $(S,\varphi)$ and $(\bbZ^d,\varphi)$ have a.s.\ the same {\em type}.
A second question we shall address in this paper is that of
establishing almost sure bounds on finite volume effective resistances
in the case of certain recurrent random walks of the type $(S,\varphi)$.
Before going to a description of our main results we discuss the main examples of
point process we have in mind.
In what follows we shall use the notation $S(\L)$ for the number of points of $S$ in any given bounded
Borel set $\L\subset \bbR^d$. For any $t>0$ and $x\in\bbR^d$ we write
$$Q_{x,t}:=\left[-\frac{t}2,\frac{t}2\right]^d\,,
\quad\; B_{x,t}=\{x\in\bbR^d:\;|x|< t\}\,,$$
for the cube with side $t$ and the open ball of radius $t$ around $x$.
To check that the models $(S,\varphi)$  are well defined, i.e.\ (\ref{pw}) is satisfied,
in all the examples described below the following simple criterion will be sufficient.

We write  $\Phi_d$,
for the class of functions $\varphi:(0,\infty)\to (0,1]$
such that $\int_0^\infty t^{d-1}\varphi(t)dt <\infty$.
Suppose the law of the point process $\bbP$ is such that
 \be\la{intg}
\sup_{x\in\bbZ^d}\bbE[S(Q_{x,1})] <\infty\,.
\end{equation}
Then it is immediate to check that $(S,\varphi)$ satisfies (\ref{pw}) for any $\varphi\in\Phi_d$.


\subsection{Examples}\la{examples}
The main example we have in mind is the case when  $\bbP$ is a {\sl
homogeneous Poisson point process} (PPP) on $\bbR^d$. In this case
we shall show that $(S,\varphi)$ and $(\bbZ^d,\varphi)$ have a.s.\
the same type, at least for the standard choices
$\varphi=\varphipa,\varphieb$. Besides its intrinsic interest as
random perturbation of lattice walks we point out that the Poisson
point process model arises naturally in statistical physics in the
study of the low-temperature conductivity of disordered
systems. In this context, the $(S,\varphieb)$ model with $\b=1$ is a
variant of the well known Mott variable--range hopping model, see
\cite{FSS} for more details. The original variable--range hopping
model comes with an environment of energy marks on top of the
Poisson point process which we neglect here since it does not
interfere with the recurrence or transience of the walk. It will be
clear that by elementary domination arguments all the results we
state for homogeneous PPP actually apply to non--homogeneous PPP
with an intensity function which is uniformly bounded from above and
away from zero.

Motivated by the variable--range hopping problem
one could consider point fields
obtained from
a crystal by dilution and spatial randomization.
 By crystal we mean any
locally finite  set $\Gamma\subset \bbR^d$ such that for a suitable
basis $v_1, v_2, \dots, v_d$ of $\bbR^d$, one has
\begin{equation}\label{paranza}
\G -x =\G \, \qquad \forall x\in G:=\bigl \{ z_1v_1+z_2v_2 +\cdots
+z_d v_d \,:\, z_i \in \bbZ \;\; \forall i  \bigr\}\,.
\end{equation}
The {\sl spatially randomized and $p$--diluted
crystal} is obtained from $\G$ by first translating  $\G$ by a
 random vector $V$   chosen with uniform distribution in the elementary cell
$$
\D=\bigl \{t_1 v_1+t_2 v_2 +\cdots+t_d v_d \,:\, 0\leq t_i <1\;\;
\forall  i \bigr\}\,,
$$
and then  erasing  each point with probability $1-p$, independently
from the others. One can check that the above construction depends
only on $\G$ and not on the particular $G$ and $\D$ chosen. In the
case of spatially randomized and $p$--diluted crystals,   $\bbP$ is
a stationary point process, i.e.\ it  is invariant w.r.t.\  spatial
translations. It is not hard to check that all the results we state
for PPP hold for any of these processes as well for the
associated Palm distributions (see \cite{FSS} for a
discussion on the Palm distribution and its relation to Mott
variable--range hopping). Therefore, to avoid lengthy repetitions we
shall not mention application of our estimates to these cases
explicitly in the sequel.

We shall also comment on applications of our results to two other classes
of point processes: {\em percolation clusters} and {\em determinantal point processes}.
We say that $S$ is a percolation cluster when $\bbP$ is the law of the
infinite cluster in super--critical Bernoulli site--percolation on
$\bbZ^d$. For simplicity we shall restrict to site--percolation but nothing changes here if one considers bond--percolation instead. The percolation cluster model has been extensively studied
in the case of nearest neighbor walks, see \cite{GKZ,BPP}. In
particular, it is well known that the simple random walk on the
percolation cluster has almost surely the same type of simple random walk on $\bbZ^d$.
Our results will allow to prove that if $S$ is the percolation cluster on $\bbZ^d$
then $(S,\varphi)$ has a.s.\ the same type of $(\bbZ^d,\varphi)$, at least for the standard choices
$\varphi=\varphipa,\varphieb$.

Determinantal point processes (DPP) on the other hand are defined as follows, see
\cite{So,Quattro} for recent insightful reviews on DPP.
Let $\cK$ be a locally trace class self--adjoint operator on $L^2(\bbR^d, dx)$.
If, in addition, $\cK$ satisfies $0\leq \cK\leq 1$ we can speak of the DPP associated to $\cK$.
Let $\bbP,\bbE$ denote the associated law and expectation.
It is always possible to associate a kernel $K(x,y)$ to $\cK$ such that
for any bounded measurable set $B\subset \bbR^d$ one has
\be\la{traccia}
\bbE [S(B)]=\tr (\cK 1_B) = \int_B K(x,x) dx<\infty\,
\end{equation}
where $S(B)$ is
the number of points in the set $B$ and
$1_B$ stands for multiplication by the indicator function of the set $B$, see \cite{So}.
Moreover, for any
 family of mutually disjoint subsets $D_1, D_2, \dots ,D_k \subset
 \bbR^d$ one has
\begin{equation}
\bbE \left[ \prod _{i=1}^k S(D_i) \right] = \int_{\prod_i D_i}
\rho_k (x_1,x_2, \dots , x_k) dx_1 d x_2\dots d x_k \,,
\end{equation}
where the   $k$--correlation function $\rho_k$ satisfies
$$ \rho_k (x_1, x_2, \dots, x_k) = \text{det} \left( K(x_i,x_j)
\right )_{1\leq i,j\leq k}\,.
$$
Roughly speaking, these processes are characterized by a
tendency towards repulsion between points, and
if we consider a {\em stationary} DPP,
i.e.\ the case where the kernel satisfies $K(x,y)=K(0,y-x)$, then the repulsive character
forces points to be more regularly spaced than in the Poissonian case.
A standard example is the sine kernel in $d=1$, where $K(x,y)=\frac{\sin(\pi(x-y))}{\pi(x-y)}$. Our results will imply for instance that for stationary DPP $(S,\varphi)$ and $(\bbZ^d,\varphi)$ have a.s.\ the same type if $\varphi=\varphipa$ (any $\a>0$) and if $\varphi=\varphieb$ with $\b<d$.

\subsection{Random resistor networks}\la{rrn}
  Our analysis of the transience and recurrence of the random walk
  $X_n$ will be based on
the well known resistor network representation of probabilistic
quantities associated to  reversible random walks on graphs, an
extensive discussion of which is found e.g.\ in the monographs
\cite{DS,LyonsPeres}. For the moment let us recall a few basic
ingredients of the electrical network analogy. We think of
$(S,\varphi)$ as an undirected weighted graph with vertex set $S$
and complete edge set $\{\{x,y\}\,,\; x\neq y\}$, every edge $
\{x,y\}$ having weight $\varphi(|x-y|)$. The equivalent electrical
network is obtained by connecting each pair of nodes $ \{x,y\}$ by a
resistor of magnitude $r(x,y):=\varphi(|x-y|)^{-1}$, i.e. by a
conductance of magnitude $\varphi(|x-y|)$. We point out that other
long--range  reversible random walks have already been studied (see
for example  \cite{BBK}, \cite{B}, \cite{KM}, \cite{M} and
references therein), but since the resistor networks associated to
these random walks are locally finite and not complete as in our
case, the techniques and estimates required here are very different.

 One can characterize
the transience or recurrence of $X_n$ in terms of the associated
resistor network.
Let
$\{S_n\}_{n\geq 1 }$ be an increasing  sequence of subsets $S_n
\subset S$ such that $S=\cup_{n\geq 1 }S_n$ and
let $(S,\varphi)_n$ denote the network obtained by collapsing all
sites in $S_n^c=S\setminus S_n$ into a single site $z_n$ (this
corresponds to the network where all resistors between nodes in
$S_n^c$ are replaced by infinitely conducting wires but all other
wires connecting $S_n$ with $S_n$ and $S_n$ with $S_n^c$ are left
unchanged). For $x \in S$ and $n$ large enough such that $x \in
S_n$,  let $R_n(x)$ denote the effective resistance between the
nodes $x$ and $z_n$ in the network $(S,\varphi)_n$. We recall that
$R_n(x)$ equals the inverse of the effective conductivity $C_n (x)$,
defined as the current flowing in the network when a unit voltage is
applied across the nodes $x$ and $z_n$. On the other hand it is well
known that $w_S(x)R_n(x)$ equals the expected number of  visits to
$x$ before exiting the set $S_n$ for our original random walk
$(S,\varphi)$ started at $x$. The sequence $R_n(x)$ is
non--decreasing and its limit $R(x)$ is called the effective
resistance of the resistor network $(S, \varphi)$ between $x$ and
$\infty$. Then, $w_S(x)R(x)=\lim_{n\to\infty}w_S(x)R_n(x)$ equals
the expected number of  visits  to $x$ for the walk $(S,\varphi)$
started in $x$,
and the walk $(S,\varphi)$ is recurrent iff $R_n(x)\to \infty$ for
some (and therefore any) $x\in S$.
In light of this, we shall
investigate the rate of divergence of $R_n(x)$ for specific
recurrent models.

Lower bounds on $R_n(x)$ can be obtained by the following
variational characterization of the effective conductivity $C_n(x)$:
\begin{equation}\label{var_con}
C_n(x)= \inf _{\substack{ h: S\rightarrow [0,1]\\ h(x)=0\,, h \equiv
1 \text{ on } S_n^c }} \frac{1}{2}\sum_{\substack{y, z \in
S\\y\not=z}} \varphi(|y-z|) \bigl(h(y)-h(z)\bigr)^2\,.
\end{equation}
The above infimum is attained when $h$ equals the electrical
potential, set to be zero on $x$ and $1$ on $S_n^c$. From
\eqref{var_con} one derives Rayleigh's monotonicity principle: the
effective conductivity $C_n(x)$ decreases whenever $\varphi$ is
replaced by $\varphi'$ satisfying $\varphi'(t)\leq \varphi(t)$ for
all $t>0$.
 Upper bounds on
$R_n(x)$ can be obtained by means of fluxes.
We recall that, given a point $x \in S$ and a subset $B\subset S$
not containing $x$, a unit flux from $x$ to $B$ is any antisymmetric
function $f:S\times S \rightarrow \bbR$ such that
$$ \text{div} f (y):= \sum _{z \in S} f(y,z)
\begin{cases}= 1 & \text{ if  } y=x\,,\\
= 0 & \text{ if } y \not =x \text{ and } y \not \in B\,,\\
\leq 0 &\text{ if  } y \in B\,.
\end{cases}
$$
If $B=\emptyset$ then $f$ is said to be a unit flux from $x$ to
$\infty$.  The energy $\cE(f)$ dissipated by  the flux $f$ is
defined as
\begin{equation}\label{energico}
\cE (f)= \frac{1}{2} \sum_{\substack{y,z \in S\\y\not=z} } r(y,z)
f(y,z)^2 \,.
\end{equation}
To emphasize the dependence on $S$ and $\varphi$ we shall often call $\cE(f)$ the
$(S,\varphi)$--energy.
Finally, $R_n(x)$, $R(x)$ can be shown to satisfy the following variational principles:
\begin{align}
 R_n(x)& = \inf \left\{ \cE(f)\,:\, f\text{ unit flux from $x$ to
$S_n^c$} \right\}\,,\label{var_res}\\
  R (x) & = \inf \left\{
\cE(f)\,:\, f\text{ unit flux from $x$ to $\infty$}
\right\}\,.\label{var_res_tot}
\end{align}
In particular, one has  the so called Royden--Lyons criterion
\cite{TLyons} for reversible random walks:
the random walk $X_n$ is transient if and only if
there exists a unit flux on the resistor network  from some
point $x \in S$ to $\infty$ having finite energy.

An immediate consequence of these facts is the following comparison tool, that we shall often use in the sequel.
\begin{Le}\la{lemmino}
Let $\bbP,\bbP'$ denote two point processes
on $\bbR^d$ such that
 $\bbP$ is stochastically dominated by
$\bbP'$ 
and let $\varphi,\varphi':(0,\infty)\to (0,\infty)$ be
such that $\varphi\leq C\varphi'$ for some constant $C>0$.
Suppose further that (\ref{pw}) is satisfied for both $(S,\varphi)$
and $(S',\varphi')$, where $S,S'$ denote the random sets
distributed according to $\bbP$ and $\bbP'$, respectively. The
following holds:
\begin{enumerate}
\item
 if $(S,\varphi)$ is transient $\bbP$--a.s., then $(S', \varphi')$ is transient $\bbP'$--a.s.
\item
if $(S',\varphi')$ is recurrent $\bbP'$--a.s., then $(S, \varphi)$
is recurrent $\bbP$--a.s.
\end{enumerate}
\end{Le}
\proof The stochastic domination assumption is equivalent to the existence of
a coupling of $\bbP$ and $\bbP'$ such that, almost surely,
$S\subset S'$ (see
e.g.\ \cite{GK} for more details). If $(S,\varphi)$ is transient then there exists a flux $f$ on $S$
with finite $(S,\varphi)$--energy from some $x\in S$ to infinity.
We can lift $f$ to a flux on $S'\supset S$ (from the same $x$ to infinity) by setting it equal to
$0$ across pairs $x,y$ where either $x$ or $y$ (or both) are not in $S$. This has finite
$(S',\varphi)$-energy, and since $\varphi\leq C\varphi'$ it will have finite $(S',\varphi')$--energy.
This proves (1). The same argument proves
(2) since if $S\subset S'$ were such that $(S,\varphi)$ is
transient then $(S',\varphi')$ would be transient and
we would have a contradiction.  \qed

\subsection{General results}
Recall the notation
$B_{x,t}$ for the open ball in $\bbR^d$ centered at $x$ with radius
$t$ and define the
function $\psi:(0,\infty)\rightarrow [0,1]$ by
\begin{equation}\label{psipsi} \psi (t):= \sup _{x\in \bbZ^d} \bbP
\bigl( S\bigl( B_{x,t}\bigr)=0\bigr)\,. 
\end{equation}

\begin{Th}\label{th1}

\smallskip

\noindent
(i) Let $d\geq 3$ and $\a>0$, or $d=1,2$ and $0<\a<d$. Suppose that
$\varphi\in\Phi_d$ and
\begin{align}
& \varphi (t) \geq c \,\varphi _{p, \a } (t)\,,\label{trans1}\\
&  \psi(t) \leq C \, t^{-\g}
 \,, \qquad\forall t >0  \,, \label{trans2}
\end{align}
for some positive constants $c,C$ and $\g>3d+\a$. Then, $\bbP$--a.s.
$(S,\varphi)$ is transient.

\smallskip

\noindent
(ii) Suppose that $d\geq 3$ and
\be\la{trans20}
\int_0^\infty \nep{a\, t^\b}\psi(t)\,dt<\infty\,,
\end{equation}
for some $a,\b>0$. Then there exists $\d=\d(a,\b)>0$ such that
$(S,\varphi)$ is a.s.\ transient whenever $\varphi (t)\geq
c\,\nep{-\d\,t^{\b}}$ for some $c>0$.

\smallskip

\noindent
(iii) Set $d\geq 1$ and
suppose that
\be\la{rick2}
 \sup_{x\in \bbZ^d}\bbE \left[ S( Q_{x,1})^2\right]<\infty\,.
\end{equation}
Then $(S,\varphi)$ is $\bbP$--a.s.\ recurrent whenever
$(\bbZ^d,\varphi_0)$ is recurrent, where $\varphi_0$ is given by
\be\la{fi0}
\varphi_0(x,y):=\max_{u\in Q_{x,1},\,v\in Q_{y,1}}\varphi(u,v)\,.
\end{equation}
\end{Th}

The proof of these general statements is given is Section \ref{gen}.
It relies on rather elementary arguments not far from the {\em rough embedding} method described in \cite[Chapter 2]{LyonsPeres}. In particular, to prove (i) and (ii) we shall construct a flux on $S$ from a point $x\in S$ to infinity and show that
it has finite $(S,\varphi)$--energy under suitable assumptions. The flux will be constructed
using comparison with suitable
long--range random walks $\bbZ^d$.  Point (iii) of Theorem
\ref{th1} is obtained by exhibiting a candidate for the electric potential in the network $(S,\varphi)$ which produces a vanishing conductivity. Again the construction is achieved using
comparison with long--range random walks on $\bbZ^d$.
%
%

Despite the simplicity
of the argument, Theorem \ref{th1} already captures
non--trivial facts such as e.g.\ the transience of the super--critical
percolation cluster in dimension two with $\varphi=\varphipa$, $\a<2$. More generally,
combining (i) and (iii) of Theorem \ref{th1} we shall obtain the following corollary.

\begin{Cor}\la{cor1}
Fix $d\geq 1$. Let $\bbP$ be one of the following point processes: a homogeneous
PPP;
the infinite cluster in super--critical  Bernoulli
site--percolation on $\bbZ^d$; a stationary DPP
on $\bbR^d$.
Then $(S,\varphipa)$ has a.s.\ the same type as $(\bbZ^d,\varphipa)$, for all $\a>0$.
%
\end{Cor}

We note that for the transience results (i) and (ii) we only need to check the
sufficient conditions (\ref{trans2}) and (\ref{trans20}) on the function $\psi(t)$.
Remarks on how to prove bounds on $\psi(t)$ for
various processes are given in Subsection \ref{prcor1}.
Conditions (\ref{trans2}) and (\ref{trans20})
in Theorem \ref{th1} are in general far from optimal.
We shall give a bound that improves point (i)
in the case $d=1$, see Proposition \ref{treno1} below.

The limitations of Theorem \ref{th1} become more important when $\varphi$ is rapidly
decaying and $d\geq 3$. For instance, if $\bbP$ is the law of the infinite
percolation cluster, then $\psi(t)$ satisfies a bound of the form
$\nep{-c\,t^{d-1}}$, see Lemma \ref{clusterlemma} below.
Thus in this case point (ii) would only allow to conclude that
  there exists $a=a(p)>0$ such that, in $d\geq 3$,
$(S,\varphi)$ is $\bbP$--a.s.\ transient if $\varphi(t)\geq C
\nep{-a\,t^{d-1}}$. However, the well known Grimmett--Kesten--Zhang
theorem about the transience of {\em nearest neighbor} random walk
on the infinite cluster in $d\geq 3$ (\cite{GKZ}, see also
\cite{BPP} for an alternative proof) together with Lemma
\ref{lemmino} immediately implies that $(S,\varphi)$ is a.s.\
transient for any $\varphi\in\Phi_d$. Similarly, one can use
stochastic domination arguments to improve point (ii) in Theorem
\ref{th1} for other processes. To this end we say that the process
$\bbP$ {\em dominates (after coarse--graining)
super--critical Bernoulli site--percolation} if $\bbP$ is such that
for some $L\in\bbN$ the random field \be\la{sil} \s= \bigl( \s(x)
\,:\; x \in \bbZ^d\bigr)\,, \quad\s (x) :=
 \chi \bigl( S ( Q_{Lx, L})\geq 1 \bigr)\,,\end{equation}
stochastically dominates the i.i.d.\ Bernoulli field on $\bbZ^d$ with some super--critical parameter
$p$. Here $\chi(\cdot)$ stands for the indicator function of an event.
In particular, it is easily checked that any homogeneous
PPP dominates  super--critical Bernoulli site--percolation.
For DPP defined on $\bbZ^d$ stochastic domination w.r.t.\ Bernoulli can be obtained under suitable hypothesis on the kernel $K$, see \cite{LyonsSteif}.
We are not aware of analogous conditions in
the continuum
that would imply that DPP dominates  super--critical Bernoulli site--percolation. In the latter cases we have to content ourselves with
point (ii) of Theorem \ref{th1} (which implies point 3 in Corollary \ref{cor100} below). We summarize our conclusions for $\varphi=\varphieb$
in the following
\begin{Cor}\la{cor100}

\smallskip
\noindent
1. Let $\bbP$ be any of the processes considered in Corollary
\ref{cor1}. Then $(S,\varphieb)$ is a.s.\ recurrent in $d=1,2$, for any $\b>0$.

\smallskip
\noindent
2. Let $\bbP$ be the law of the
infinite cluster in super--critical Bernoulli site--percolation on
$\bbZ^d$ or a homogeneous PPP or any other process that
dominates  super--critical Bernoulli site--percolation.
Then $(S,\varphieb)$ is a.s.\ transient in $d\geq 3$, for any $\b>0$.

\smallskip
 \noindent
3. Let $\bbP$ be any stationary DPP.
Then $(S,\varphieb)$ is $\bbP$--a.s.\ transient in $d\geq 3$, for any $\b\in(0, d)$.
\end{Cor}

We point our that, by the same proof, point 2) above remains
true if $(S,\varphieb)$ is replaced by  $(S,\varphi)$, $\varphi \in
\Phi_d$.

\subsection{Bounds on finite volume effective resistances}
When a network $(S,\varphi)$ is recurrent the effective resistances
$R_n(x)$ associated to the finite sets   $S_n:=S\cap [-n,n]^d $
diverge, see (\ref{var_res}), and we may be interested in obtaining
quantitative information on their growth with $n$. We shall consider
in particular the case of point processes in dimension $d=1$, with
$\varphi=\varphipa$, $\a\in[1,\infty)$, and the case $d=2$ with
$\varphi=\varphipa$, $\a\in[2,\infty)$. By Rayleigh's monotonicity
principle, the bounds given below apply also to $(S,\varphi)$,
whenever $\varphi \leq C \varphipa$. In particular, they cover the
stretched exponential case $(S,\varphieb)$.

We say that the point process $\bbP$ is {\em dominated by an i.i.d.\ field} if the
following condition holds: There exists
$L\in\bbN$ such that the random field
$$N_L= \bigl( N(v) \,:\;
v\in \bbZ^d\bigr)\,, \quad N (v) :=
 S ( Q_{Lv, L})\,,$$
is stochastically dominated by independent non--negative random variables
$\{\G_v,\;v\in\bbZ^d\}$ with finite expectation.

For the results in $d=1$ we shall require the following exponential moment condition
on the dominating field $\G$:
There exists $\e>0$ such that \be\la{expe}
\bbE[\nep{\e\,\G_v}]<\infty\,.
\end{equation}
For the results in $d=2$ it will be sufficient to require the existence of the
fourth moment:
\be\la{fourth}
\bbE\left[\G_v^4\right]<\infty\,.
\end{equation}

It is immediate to check that any homogeneous PPP is dominated by an i.i.d.\ field
in the sense described above and the dominating field $\G$ satisfies (\ref{expe}).
Moreover, this continues to hold for non--homogeneous Poisson process with a uniformly bounded intensity function. We refer the reader to \cite{LyonsSteif,GK} for examples of determinantal
processes satisfying this domination property.


\begin{Th}\label{fegato}
Set $d=1$, $\varphi = \varphipa$ and $\a\geq 1$. Suppose that the
point process $\bbP$ is dominated by an i.i.d.\ field satisfying
(\ref{expe}). Then, for $\bbP$--a.a.\ $S$ the network $(S
,\varphi)$ satisfies:  given $x\in S$   there exists a constant
$c>0$ such that
\begin{equation}\la{rn1}
R_n(x) \geq  c
\begin{cases} \log n  & \text{ if } \a=1\,,\\
n^{\a-1} & \text{ if } 1<\a<2\,,\\
n/\log n & \text{ if } \a=2,,\\
n & \text{ if } \a>2\,,
\end{cases}
\end{equation}
for all $n\geq 2$ such that $x \in S_n$.
\end{Th}

\begin{Th}\la{combine}
Set $d=2$,  $\varphi = \varphipa$ and $\a\geq 2$. Suppose that
$\bbP$ is dominated by an i.i.d.\ field satisfying (\ref{fourth}).
Then, for $\bbP$--a.a.\ $S$ the network $(S ,\varphi)$
satisfies: given  $x\in S$ there exists a constant $c>0$ such that
\begin{equation}\la{rn0}
R_n (x) \geq c
\begin{cases} \log n & \text{ if } \a >2\,,\\
\log (\log n) & \text{ if } \a=2\,,
\end{cases}
\end{equation}
for all $n\geq 2$ such that $x \in S_n$.
\end{Th}

The proof of Theorem \ref{fegato} and Theorem \ref{combine} is given
in Section \ref{coraggio}. The first step is to reduce the network
 $(S ,\varphi)$ to a simpler network by using the domination
assumption. In the proof of Theorem \ref{fegato} the effective
resistance of this simpler network is then estimated using the
variational principle (\ref{var_con}) with suitable trial functions.

In the proof of Theorem \ref{combine} we are going to exploit a further reduction of the
network which ultimately leads to a one--dimensional nearest neighbor network where effective
resistances are easier to estimate. This construction uses an idea already appeared in \cite{K}, see also \cite{CF2} and \cite{ABS} for recent applications,
which allows to go from long--range to nearest neighbor networks, see Section \ref{coraggio} for
the details. Theorem \ref{combine} could be also proved using the variational principle
(\ref{var_con}) for suitable choices of the trial function, see the remarks in Section \ref{coraggio}.

It is worthy of note that the proofs of these results are
constructive in the sense that they do not rely on results already
known for the corresponding $(\bbZ^d,\varphipa)$ network. In
particular, the method can be used to obtain quantitative lower
bounds on $R_n(x)$ for the deterministic case $S\equiv \bbZ^d$,
which is indeed a special case of the theorems.  In the latter case
the lower  bounds obtained   here as well suitable upper bounds  are
probably well known but we were not able to find references to that
in the literature. In appendix \ref{losangeles}, we show how to
bound from above   the effective resistance $R_n(x)$ of the network
$(\bbZ^d,\varphipa)$ by means of harmonic analysis. The resulting
upper bounds match with the lower bounds of Theorems \ref{fegato}
and \ref{combine}, with exception of the case $d=1$, $\a=2$ where
our upper and lower bounds differ by a factor $\sqrt{\log n}$.


\subsection{Constructive proofs of transience}
While the transience criteria summarized in Corollary \ref{cor1} and
Corollary \ref{cor100} are based on known results for the
deterministic networks $(\bbZ^d,\varphi)$ obtained by
classical harmonic analysis, it is possible to give constructive
proofs of these results by exhibiting explicit fluxes with finite
energy on the network under consideration. We discuss two results
here in this direction. The first gives an improvement over the
criterium in Theorem \ref{th1}, part (i), in the case $d=1$. This
can be used in particular to give a ``flux--proof'' of the well known
fact that $(\bbZ,\varphipa)$ is transient for $\a<1$. The second
result gives a constructive proof of transience of a deterministic
network, which, in turn, reasoning as in the proof of Theorem
\ref{th1} part (i), gives a flux--proof that $(\bbZ^2,\varphipa)$
is transient for $\a<2$.

%
In order to state the one-dimensional result, it is convenient
to number   the points of $S$  as $S=\{x_i\} _{i \in I}$ where $x_i
<x_{i+1}$, $x_{-1}< 0 \leq x_0$ and $\bbN\subset I $ or
$-\bbN \subset I$ (we assume that $|S|=\infty$,
$\bbP$--a.s., since otherwise the network is recurrent).
For simplicity of notation we assume below that
$\bbN\subset I$, $\bbP$--a.s. The following result can be easily
extended to the general case by considering separately the
conditional probabilities $\bbP(\cdot | \bbN \subset I)$ and
$\bbP(\cdot |\bbN\not \subset I)$,  and applying a symmetry argument
in the second case.
\begin{Pro}\label{treno1}
Take  $d=1$ and $\a\in (0,1)$. Suppose that for some positive
constants $c,C$ it holds
\begin{align}
& \varphi (t) \geq c \varphipa (t)\,, \qquad  t>0 \,,\\
& \bbE \bigl (\, |x_n-x_k|^{1+\a}\, \bigr)\leq C \,(n-k)^{1+\a}\,,
\qquad \forall n>k\geq 0\,. \label{momenti}
\end{align}
Then $\bbP$--a.s. $(S,\varphi)$ is transient. In particular, if
$\bbP$ is a renewal point process such that
\begin{equation}\label{momentibis}
\bbE (|x_1-x_0|^{1+\a}) <\infty \,,\end{equation} then $\bbP$--a.s.
$(S,\varphi)$ is transient.
\end{Pro}


Suppose that $\bbP$ is a renewal point process and write
$$\wt \psi (t):=
\bbP(x_1-x_0 \geq t)\,.
$$
Then (\ref{momentibis}) certainly holds as soon as e.g.\ $\wt \psi$
satisfies $\wt\psi (t) \leq C t^{-(1+\a+\e)}$ for some positive
constants $C,\e$. We can check that this improves substantially over
the requirement in Theorem \ref{th1}, part (i), since if $\psi$ is
defined by (\ref{psipsi}), then we have, for all  $t>1$:
$$ \tilde \psi (2t) = \bbP ( S \cap B_{x_0+t,t}  = \emptyset )
\leq \psi (t-1) \,.$$

\smallskip

The  next  result concerns the deterministic two-dimensional
network $(S_*,\varphipa)$ defined as follows.
%
Identify $\bbR^2$ with the complex
plane $\bbC$, and define the set $ S_*:= \cup_{n=0}^\infty C_n$,
where
\begin{equation}\label{def_cn}
C_n:= \left\{ n e^{k\frac{2i\pi}{n+1}} \in {\mathbb C} :\: k\in \{0,
\dots, n\} \right\}\,.
\end{equation}
%
\begin{Th}\label{alex}
The network $(S_*, \varphipa)$ is transient for all $\a\in(0,2)$.
\end{Th}
This theorem, together with the comparison techniques
developed in  the next section (see Lemma  \ref{protra} below),
allows to recover by a flux--proof the transience of $(\bbZ^2,
\varphipa)$ for $\a\in (0,2)$. The proofs of Proposition
\ref{treno1} and Theorem \ref{alex} are given in Section
\ref{massimo}.

\section{Recurrence and transience by comparison methods}\la{gen}
Let $S_0$ be a given locally finite subset of $\bbR^d$  and
 let
$(S_0,\varphi_0)$ be the associated random walk. We assume that
$w_{S_0}(x)<\infty$ for all $x\in S_0$ and that $\varphi_0(t)>0$ for
all $t>0$. Recall that in the resistor network picture every node
$\{x,y\}$ is given the resistance $r_0(x,y):=\varphi_0(|x-y|)^{-1}$.
To fix ideas we may think of $S_0=\bbZ^d$ and either
$\varphi_0=\varphipa$ or $\varphi_0=\varphieb$. $(S_0,\varphi_0)$
will play the role of the deterministic background network.

Let $\bbP$ denote a simple point process on $\bbR^d$, i.e.\ a
probability measure on the set $\O$ of locally finite subsets $S$ of
$\bbR^d$, endowed with the $\si$--algebra $\cF$ generated by the
counting maps $N_\L:\O\to\bbN\cup\{0\}$, where $N_\L(S)= S(\L)$ is the number of points
of $S$ that belong to $\L$ and
$\L$ is a bounded Borel subset
of $\bbR^d$. We shall use $S$ to denote a generic random
configuration of points distributed according to $\bbP$. We assume
that $\bbP$ and $\varphi$ are such that (\ref{pw}) holds.

Next, we introduce a map $\phi: S_0\to S$, from our reference set $S_0$ to the
random set $S$.
For any $x\in S_0$ we write $\phi(x)=\phi(S,x)$
for the point in $S$ which is closest to $x$ in Euclidean distance.
If the Euclidean distance from $x$ to $S$ is minimized by more than
one point in $S$ then choose $\phi(x)$ to be the one with lowest
lexicographic order. This defines a measurable map
$\O\ni S\to \phi(S,x)\in\bbR^d$ for every $x\in S_0$.

For any point $u\in S$ define the {\em cell}
$$
V_u:=\{x\in S_0\,:\;\, u=\phi(x)\}\,.
$$
By construction $\{V_u\,,\;u\in S\}$ determines a
partition of the original vertex set $S_0$.
Clearly, some of  the $V_{u}$ may be
empty, while some may be large (if $S$ has large ``holes'' with respect to
$S_0$). Let $N(u)$ denote the number of points (of $S_0$) in the cell
$V_u$. We denote by $\bbE$ the expectation with respect to $\bbP$.

\begin{Le}\la{protra}
Suppose $(S_0,\varphi_0)$ is transient. If there exists $C<\infty$
such that  for all $x\not=y$ in $ S_0$  it holds \be\la{bound}
\bbE\left[N(\phi(x)) N(\phi(y))\, r(\phi(x),\phi(y)) \right]\leq C\,
r_0(x,y)\,,
\end{equation}
then $(S,\varphi)$ is $\bbP$--a.s.\ transient.
\end{Le}
\proof Without loss of generality we shall assume that $0\in
S_0$. Since $(S_0,\varphi_0)$ is transient, from the Royden--Lyons
criterion recalled in Subsection \ref{rrn}, we know that there
exists a unit flux $f:S_0\times S_0\to \bbR$ from $0\in S_0$ to
$\infty$ with finite $(S_0,\varphi_0)$-energy.
By the same criterion, in order to prove the transience of
$(S,\varphi)$  we only need to exhibit a unit flux from some point
of $S$ to $\infty$ with finite $(S,\varphi)$-energy. To this end,
for any $u,v\in S$ we define
$$
\theta(u,v) = \sum_{x\in V_{u}}\sum_{y\in V_{v}} f(x,y)\,.
$$
  If either $V_u$ or $V_v$ are empty we set $\theta(u,v)=0$. Note that the above sum
  is finite for all $u,v\in S$,  $\bbP$--a.s. Indeed condition (\ref{bound}) implies that
  $N(\phi (x) )< \infty$ for all $x\in S_0$, $\bbP$--a.s.
   Thus,
$\theta$  defines a unit flux from $\phi(0)$ to infinity on
$(S,\varphi)$. Indeed, for every $u,v\in S$ we have $\theta  (u,v)=-
\theta  (v,u)$ and for every $u\neq \phi( 0)$ we have $\sum_{v\in
S}\theta(u,v)=0$. Moreover,
$$
\sum_{ v\in S} \theta(\phi(0),v) = \sum_{x\in V_{\phi(0)}}\sum_{y\in
S_0 } f(x,y) = \sum_{\substack{x\in V_{\phi(0)}:\\x\not=0}
}\sum_{y\in S_0} f(x,y) + \sum_{y\in S_0} f(0,y)=0+1=1\,.
$$
The energy of the flux $\theta$ is given by \be\la{enflux}
\cE(\theta):=\frac12\sum_{u\in S}\sum_{v\in S} \theta(u,v)^2 r(u,v)\,.
\end{equation}
From Schwarz' inequality  $$
\theta(u,v)^2\leq N(u)N(v)\sum_{x\in V_u}\sum_{y\in V_v} f(x,y)^2\,.$$
It follows that
\be\la{enflux1}
\cE(\theta)\leq
\frac12\sum_{x\in S_0}\sum_{y\in S_0}
f( x, y)^2
\,N({\phi(x)})\,N({\phi(y)})\,
r(\phi(x),\phi(y))\,.
\end{equation}
Since $f$ has finite energy on $(S_0,\varphi_0)$ we see that
condition (\ref{bound}) implies $\bbE[\cE(\theta)]<\infty$. In
particular, this shows that $\bbP$--a.s.\ there exists a unit flux
$\theta$ from some point $u_0\in S$ to $\infty$ with finite $(S,\varphi)$-energy.  \qed

\medskip

To produce an analogue of Lemma \ref{protra} in the recurrent case
we introduce the set $\wt S=S\cup S_0$ and consider the network
$(\wt S,\varphi)$. From monotonicity of resistor networks,
recurrence of $(S,\varphi)$ is implied by recurrence of $(\wt
S,\varphi)$. We define the map $\phi': \wt S\to S_0$, from $\wt S$
to the reference set $S_0$ as the map $\phi$ introduced before, only
with $S_0$ replaced by $\wt S$ and $S$ replaced by $S_0$. Namely,
given $x \in \wt S$ we define  $\phi ' (x)$ as the point in $S_0$
which is closest to $x$ in the Euclidean distance (when there is
more than one minimizing point, we  take the one with lowest
lexicographic order).  Similarly, for any point $x\in S_0$ we define
$$
V'_x:=\{u\in \wt S\,:\;\, x=\phi'(u)\}\,.
$$
Thus $\{V'_x\,,\;x\in S_0\}$ determines a partition of $\wt S$.
Note that in this case all $V'_x$ are non--empty ($V'_x$ contains
$x\in \wt S$).

As an example,  if $S_0=\bbZ^d$ then $\phi '(x)$, $x \in \wt S$, is the
only point in $\bbZ^d$ such that $x \in \phi'(x)+ (-1/2,1/2]^d$, while
$V'_x = \wt S \cap ( x+(-\frac{1}{2},\frac{1}{2}]^2)$ for any $x \in
\bbZ^d$.


\begin{Le}\la{prorec}
Suppose that  $(S_0,\varphi_0)$ is recurrent and that
$\bbP$--a.s.\ $V_x'$ is finite for all $x \in S_0$.  If there exists
$C<\infty$ such that for all $x\not=y$ in $ S_0$ it holds
\be\la{boundo} \bbE\Bigl[\sum_{u\in V'_{x}}\sum_{v\in V'_{y}}
\,\varphi(|u-v|) \Bigr]\leq C\, \varphi_0 (|x-y|)\,,
\end{equation}
then $(S,\varphi)$ is $\bbP$--a.s.\ recurrent.
\end{Le}
\proof Without loss of generality we shall assume that $0\in
S_0$. Set $S_{0,n}=S_0\cap[-n, n]^d$, collapse all sites in
$S_{0,n}^c=S_0\setminus S_{0,n}$ into a single site $z_n$ and call
$c(S_{0,n})$ the effective conductivity between $0$ and $z_n$, i.e.\
the net current flowing in the network when a unit voltage is
applied across $0$ and $z_n$. Since $(S_0,\varphi_0)$ is recurrent
we know that $c(S_{0,n})\to 0$, $n\to\infty$.

Recall that $c(S_{0,n})$ satisfies
\begin{equation}\label{elefantino}
c(S_{0,n})= \frac12 \sum_{x,y\in S_0} \varphi_0(|x-y|)
(\psi_n(x)-\psi_n(y))^2\,,
\end{equation}
where $\psi_n$ is the electric potential, i.e.\ the unique function on $S_0$
that is harmonic in $S_{0,n}$, takes the value $1$ at $0$ and vanishes out of $S_{0,n}$.

Given $S\in\O$, set
$$\wt S_n = \cup_{x\in S_{0,n}} V'_x\,.$$
Note that $\wt S_n$ is an increasing sequence of finite sets,
covering all $S$. Collapse all sites in $\wt S_{n}^c$ into a single
site $\wt z_n$ and call $c(\wt S_{n})$ the effective conductivity
between $0$ and $\wt z_n$ (by construction $0\in \wt S_n$). From the
Dirichlet principle \eqref{var_con}  we have
$$
c(\wt S_{n}) \leq  \frac12 \sum_{u,v\in \wt S} \varphi(|u-v|) (g(u)-g(v))^2\,,
$$
for any $g:\wt S\to [0,1]$ such that $g(0)=1$ and $g=0$ on $\wt S_{n}^c$.
Choosing $g(u)=\psi_n(\phi'(u))$ we obtain
$$
c(\wt S_{n}) \leq  \frac12 \sum_{x,y\in S_0}
(\psi_n(x)-\psi_n(y))^2\,\sum_{u\in V'_x}\sum_{v\in V'_y}
\,\varphi(|u-v|)\,.
$$
From the assumption (\ref{boundo}) and the recurrence of
$(S_0,\varphi_0)$ implying that (\ref{elefantino}) goes to zero, we
deduce that
$\bbE [c(\wt S_{n})] \to  0$, $n\to\infty$. Since $c(\wt S_{n})$ is 
monotone decreasing we deduce that $c(\wt S_{n})\to 0$, $\bbP$--a.s.
This implies $\bbP$--a.s. recurrence of $(\wt S,\varphi)$ and the claim follows.
\qed


\smallskip

\subsection{Proof of Theorem \ref{th1}}\la{prth1}
We first prove part (i) of the
theorem, by applying the general statement derived in Lemma
\ref{protra} 
in the case $S_0=\bbZ^d$ and $\varphi_0=\varphipa$.
Since  $(S_0, \varphipa)$ is transient whenever $d\geq 3$, or
$d=1,2$ and $0<\a<d$, we only need to verify condition
(\ref{bound}). For the moment we only suppose that $\psi (t) \leq C'
t ^{-\g}$ for some $\g>0$.

Let us fix $p,q>1$ s.t. $1/p+1/q=1$.
Using H\"{o}lder's
inequality and then
Schwarz' inequality, we obtain
\begin{align}\label{HS}
&\bbE\left[N(\phi(x)) N(\phi(y))\, r(\phi(x),\phi(y)) \right]\\
&\qquad\qquad\qquad\leq
\bbE \left[ N (\phi(x))^{2q} \right]^{\frac{1}{2q} } \bbE \left[ N (\phi(y))^{2q} \right]^{\frac{1}{2q} }\bbE\left[ r
(\phi(x),\phi(y) )^p\right]^{\frac{1}{p}}
\nonumber\end{align}
for any $x\not=y$ in $\bbZ^d$.  By assumption (\ref{trans1}) we know
that
\begin{equation}\label{samarcanda}
r(\phi(x),\phi(y) ) ^p \leq
c \, r_{p,\a} (\phi(x), \phi (y))^p :=
c
\, \bigl(1\lor |\phi(x) -\phi(y) |^{p(d+\a)}\bigr)\,.
\end{equation}
We shall use $c_1,c_2,\dots$ to denote constants independent from $x$ and $y$ below.
From the triangle inequality 
$$
|\phi(x) -\phi(y) |^{p(d+\a)} \leq c_1 \left( |\phi(x) -x
|^{p(d+\a)}+|x -y |^{p(d+\a)}+|\phi(y) -y |^{p(d+\a)}\right)\,.
$$
From (\ref{samarcanda}) and the fact that $|x-y|\geq 1$ we derive
that
\begin{equation}\label{fiorello}
\bbE\left[r(\phi(x),\phi(y) ) ^p  \right]
\leq c_2 \sup_{z \in \bbZ^d}
\bbE\left[|\phi(z) -z |^{p(d+\a)}\right] + c_2 |x-y|^{p(d+\a)}\,.
\end{equation}
Now we observe that $|\phi(z)-z|\geq t$ if and only if $B_{z,t}\cap
S=\emptyset$. Hence we can estimate
$$
\bbE\left[|\phi(z) -z |^{p(d+\a)}\right] \leq 1 + \int _1 ^\infty
\psi\left(t^\frac1{p(d+\a)}\right) d t
\leq 1 + C\,\int_1^\infty t^{-\frac{\g}{p(d+\a)}} dt\leq c_3\,,
$$ whenever $ \g>p(d+\a)$ as we assume.
Therefore, using $|x-y|\geq 1$,
from (\ref{fiorello}) we see that for any
$x\not=y$ in $\bbZ^d$:
\begin{equation}\label{rosario}
\bbE\left[r(\phi(x),\phi(y) ) ^p  \right]^\frac1p\leq c_4\, r_{p,\alpha} (x,y)\,.
\end{equation}

\medskip

Next, we estimate the expectation $\bbE\left[
N(\phi(x))^{2q}\right]$ from above, uniformly in $x \in \bbZ^d$. To this end we shall need the following simple geometric lemma.
\begin{Le}\la{bxts}
Let $E(x,t)$ be the event that $S\cap B(x,t) \not= \emptyset $ and
$S\cap B(x\pm 3 \sqrt{d} \,t e_i,t  )\not = \emptyset $, where
$\{e_i:1\leq i \leq d\}$ is the canonical basis of $\bbR^d$. 
%
Then, on the event $E(x,t)$ we have $\phi(x) \in B(x,t)$, i.e.
$|\phi(x) -x|< t$, and $z \not \in V_{\phi(x)}$ for all $z \in
\bbR^d$ such that $|z-x|>  9 d\sqrt{d} \,t $
\end{Le}

Assuming for a moment the validity of Lemma \ref{bxts} the proof
continues as follows. From Lemma \ref{bxts} we see that, for a
suitable constant $c_5$, the event $N(\phi(x) ) >  c_5 t^d$ implies
that at least one of the  $2d+1$ balls $B(x,t)$, $B(x\pm 3
\sqrt{d}  t e_i ,t)$ must have empty intersection with $S$. Since
 $B(x\pm \lfloor 3
\sqrt{d}  t\rfloor e_i ,t-1)\subset
 B(x\pm 3
\sqrt{d}  t e_i ,t)$ for $t\geq 1$, we conclude that
$$  \bbP\left[ N(\phi(x))  > c_5 t^d\right] \leq (2d+1)
\psi(t-1)\,, \qquad t \geq  1 . 
$$
Taking $c_6$ such that $c_5^{-\frac1d}c_6^{\frac{1}{2qd}}=2$,  it
follows that
\begin{multline}\label{sera}
\bbE\left [ N(\phi(x))^{2q} \right]= \int_0^\infty
\bbP( N(\Phi(x))^{2q} > t)dt
\\
\leq c_6 +(2d+1)\int _{c_6}^\infty
\psi\left(c_5^{-\frac1d}t^{\frac{1}{2qd}}-1 \right)dt \leq c_6+
c_7\int_1^\infty t^{-\frac{\g}{2qd}}dt\leq c_8\,,
\end{multline}
as soon as $ \g>2qd$.

\medskip

Due to (\ref{HS}), (\ref{rosario}) and (\ref{sera}), the hypothesis
(\ref{bound}) of Lemma \ref{protra} is satisfied if $\psi(t)\leq
C\,t^{-\g} $ for all $t\geq 1$, where
 $\g$ is a constant satisfying
$$ \g>p(d+\a)\,, \qquad \g>2qd=\frac{2pd}{p-1}\,.$$
We observe that the functions $(1,\infty)\ni p \rightarrow p(d+\a)$
and $(1,\infty)\ni p \rightarrow \frac{2pd}{p-1}$ are respectively
increasing and decreasing and intersect in only one point
$p_*=\frac{3d+\a}{d+\a}$.
Hence, optimizing over $p$,
it is enough to require that
$$ \gamma >
\inf_{p>1} \max\{ p(d+\a), \frac{2pd}{p-1} \}= p_*(d+\a)=3d+\a\,.
$$
This concludes the proof of Theorem \ref{th1} (i).

\medskip

\noindent
{\em Proof of lemma \ref{bxts}.}
The first claim is trivial since $S\cap B(x,t)\neq\emptyset$ implies $\phi(x)\in B(x,t)$.
In order to prove the second one
we proceed as follows. For simplicity of notation we set $m:=3\sqrt{d}$ and $k:=9d$. Let us take $z \in \bbR^d$ with
$|z-x|>k\sqrt{d}\,t $. Without loss of generality, we can suppose
that $x=0$, $z_1 >0$ and $z_1 \geq |z_i|$ for all $i =2,3,\dots, d$.
Note that this implies that $k\sqrt{d} \,t < |z|\leq \sqrt{d} z_1$,
hence $z_1> k t$. Since
\begin{align*}
& \min  _{u \in B(0,t) } |z-u| =|z|-t\\
& \max _{u \in  B( m t e_1,t) } |z-u|\leq|z-mt e_1|+t\,,
\end{align*}
if we prove that
\begin{equation}\label{chepizza}
|z|-|z-mt e_1|>  2t\,,
\end{equation}
we are sure that the distance from $z$ of  each point in $S\cap
B(0,t)$ is larger than the distance   from $z$  of each point of
$S\cap B(mt e_1, t)$. Hence it cannot be that $z \in V_{\phi(0)}$.
In order to prove (\ref{chepizza}), we first observe that the map
$(0,\infty) \ni x\rightarrow\sqrt{x+a}-\sqrt{x+b}\in (0,\infty) $ is
decreasing for $a>b$. Hence we obtain that
$$|z|-|z-mt e_1|\geq \sqrt{d}z_1- \sqrt{ (z_1-mt)^2+(d-1)z_1^2}\,.
$$
Therefore, setting $x:=z_1/t$, we only need to prove that
$$
\sqrt{d}x- \sqrt{ (x-m)^2+(d-1)x^2}>2\,, \qquad \forall x>k\,.
$$
By the mean value theorem applied to the function $f(x)=\sqrt{x}$,
it must be
\begin{multline*}
\sqrt{d}x- \sqrt{ (x-m)^2+(d-1)x^2}\geq \frac{1}{2\sqrt{d} x} \left(
dx^2- (x-m)^2-(d-1)x^2 \right)=\\
\frac{ 2xm-m^2}{ 2\sqrt{d} x} > \frac{m}{\sqrt{d}}-
\frac{m^2}{k}=2\,.
\end{multline*}
This completes the proof of (\ref{chepizza}).
\qed

\bigskip

\noindent {\sl Proof of Theorem \ref{th1} Part (ii)}.
We use the same approach as in Part (i) above.
We start again our estimate from (\ref{HS}).
Moreover, as in the proof of (\ref{sera}) it is clear that hypothesis
(\ref{trans20})
implies $\bbE[N(\phi(x))^{2q}]<\infty$ for any $q>1$, uniformly in $x\in\bbZ^d$. Therefore it remains to check that
\be\la{tc1}
r_0(x,y):=\bbE\left[r(\phi(x),\phi(y))^p\right]^\frac1p \,,
\end{equation}
defines a transient resistor network on $\bbZ^d$, for any $d\geq 3$, under the assumption that
$$r(\phi(x),\phi(y))\leq C\,\nep{\d|\phi(x)-\phi(y)|^\b}\,.$$
For any $\b>0$ we
can find a constant $c_1=c_1(\b)$
such that
$$  r(\phi(x),\phi(y))\leq C\,
\exp{\left(\d\,c_1\left[|\phi(x)-x|^\b + |x-y|^\b +
|\phi(y)-y|^\b\right]\right)}\,.
$$
Therefore, using Schwarz' inequality
we have
\begin{align*}
&\bbE\left[r(\phi(x),\phi(y))^p\right]^\frac1p
\\
&\qquad\;\leq c_2\exp{\left(\d\,c_2 |x-y|^\b\right)}\,
\bbE\left[\exp{\left(\d\,c_2 |\phi(x)-x|^\b
\right)}\right]^\frac12
\bbE\left[\exp{\left(\d\,c_2 |\phi(y)-y|^\b
\right)}\right]^\frac12\,.
\end{align*}
For $\g>0$
\begin{align*}
\bbE\left[\exp{\left(\g|\phi(x)-x|^\b \right)}\right] &\leq 1 +
\int_1^\infty \psi\left(\left(\frac1\g\,\log
t\right)^\frac1\b\right)\,dt \\ & = 1 + \g \beta \,\int_0^\infty
\psi(s)\,\nep{\g s^\b}\,s^{\b-1}\,ds\,,
\end{align*}
where, using (\ref{trans20}),
the last integral is finite for $\g<a$.
Taking $\g=\d\,c_2$ and $\d$ sufficiently small we arrive at the conclusion that uniformly in $x,$
$$
\bbE\left[r(\phi(x),\phi(y))^p\right]^\frac1p
\leq c_3\exp{\left(c_3 |x-y|^\b\right)}=:\wt r_0(x,y)\,.
$$
Clearly, $\wt r_0(x,y)$
defines a transient resistor network on
$\bbZ^d$ and the same claim for $r_0(x,y)$
follows from monotonicity.
This ends the proof of Part (ii) of Theorem \ref{th1}. \qed

\bigskip

\noindent \noindent {\sl Proof of Theorem \ref{th1} Part (iii)}.
Here we use the criterion given in Lemma \ref{prorec} with
$S_0=\bbZ^d$. With this choice of $S_0$ we have that
$V_x'\subset \{x\}\cup (S\cap Q_{x,1})$, $x \in \bbZ^d$. Recalling definition (\ref{fi0}) we see that
for all $x\not = y $ in $\bbZ^d$,
$$
\bbE\left[ \sum_{u\in V_x'}
\sum_{v\in V'_y}\varphi(|u-v|)\right]
\leq c_1\,\varphi_0(x,y)\,\bbE\left [ (1+S(
Q_{x,1})) (1+S(Q_{y,1}) )\right]\,.
$$
Using the  Schwarz' inequality and condition (\ref{rick2}) the
last expression is bounded by $c_2\, \varphi _0 (x,y) $. This
implies condition (\ref{boundo}) and therefore the a.s.\ recurrence
of $(S, \varphi)$. \qed

\subsection{Proof of Corollary \ref{cor1}}\la{prcor1}
We start with some estimates on the function $\psi(t)$. Observe that
for Poisson point processes PPP($\l$) one has $\psi(t) = \nep{-\l \,t^d}$.
A similar estimate holds for a stationary DPP. More generally, for
DPP we shall use the following facts.
\begin{Le}\label{domenica}
Let $\bbP$ be a determinantal point process on $\bbR^d$ with kernel
$K$. Then the function $\psi(t)$ defined in (\ref{psipsi}) equals
\begin{equation}\label{psipsi2}
\psi(t)=\sup _{x \in \bbZ^d} \prod _i \bigl( 1 - \l_i (B(x,t)) \bigr)\,,
\end{equation}
where $\l_i(B)$ denote the eigenvalues of $\cK 1_{B}$ for any bounded Borel set $B\subset\bbR^d$.
In particular, condition (\ref{trans2}) is satisfied if
\begin{equation}\label{trans3}
\exp \Big\{ -  \int_{ B(x,t)} K(u,u) du \Big\}\leq C
t^{-\g}\,, \qquad t >0\,, \; x \in \bbZ^d\,,
\end{equation}
for some constants $C>0$ and $\g>3d+\a$.
If $\bbP$ is a stationary DPP then $\psi(t)\leq \nep{-\d\,t^d}$, $t>0$, for some $\d>0$.

Finally, condition (\ref{rick2}) reads
\begin{equation}
\label{rick3}
\sup_{x \in \bbZ^d} \Bigl\{\sum_i \l_i(Q(x,1))  +
\sum_i\sum_{j:\;j\neq i}\l_i (Q(x,1))\l_j (Q(x,1))
\Bigr\}< \infty \,.
\end{equation}
In particular, condition (\ref{rick2}) always holds if $\bbP$ is a stationary DPP.
\end{Le}
\proof
It is known (see
\cite{Quattro} and \cite{So}) that for each bounded Borel set
$B\subset \bbR^d$ the number of points $S(B)$ has the same law of
the sum $\sum_i B_i$, $B_i$'s being independent Bernoulli random
variables with parameters $\l_i (B)$.
This implies  identity (\ref{psipsi2}).
Since $1-x \leq e^{-x}$, $x\geq 0$, we can bound the r.h.s.\ of
(\ref{psipsi}) by $ \prod _i e^{- \l_i (B) }= e^{- Tr (\cK 1_B)}$.
This identity and (\ref{traccia}) imply  (\ref{trans3}). If the DPP is stationary then
$K(u,u)\equiv K(0)>0$ and therefore $\psi(t)\leq\nep{-K(0)\,t^d}$.
Finally,
(\ref{rick3}) follows from the identity
$\bbE[S(Q(x,1))^2] = \sum_i \l_i(Q(x,1))  + \sum_i\sum_{j:\;j\neq i}\l_i (Q(x,1))\l_j (Q(x,1))$,
again a consequence of the fact that $S(Q(x,1))$ is the sum of independent Bernoulli random variables with parameters $\l_i(Q(x,1))$. Since
the sum of the $\l_i(Q(x,1))$'s is finite,
$\bbE[S(Q(x,1))^2]<\infty$ for any $x\in\bbZ^d$.
If the DPP is stationary it is uniformly finite. \qed

\medskip
The next lemma allows to estimate $\psi(t)$ in the case of percolation clusters.
\begin{Le}\la{clusterlemma}
Let $\bbP$ be the law of the infinite cluster in super--critical Bernoulli site (or bond)
percolation in $\bbZ^d$, $d\geq 2$. Then there exist constants $k,\d>0$ such that
$$\nep{-\d^{-1}\,n^{d-1}}\leq \psi(n)\leq k\,\nep{-\d\,n^{d-1}}\,,\quad n\in\bbN\,.$$
\end{Le}
\proof The lower bound follows easily by considering the event that
e.g.\ the cube centered at the origin with side $n/2$ has all the
boundary sites (or bonds) unoccupied. To prove the upper bound one
can proceed as follows. Let $K_n(\g)$, $\g>0$, denote the event that
there exists an open cluster $C$ inside the  box
$B(n)=[-n,n]^d\cap\bbZ^d$ such that $|C|\geq \g\,n^d$. Known
estimates (see e.g.\ Lemma (11.22) in Grimmett's book \cite{G} for
the case $d=2$ and Theorem 1.2 of Pisztora's \cite{Pisztora} for
$d\geq 3$) imply that there exist constants $k_1,\d_1,\g>0$ such
that \be\la{pisz} \bbP(K_n(\g)^c)\leq k_1\,\nep{-\d_1\,n^{d-1}}\,.
\end{equation}
On the other hand, let $C_x$ denote the open cluster at $x\in\bbZ^d$ and write $\cC_\infty$ for the infinite open cluster.
From \cite[Theorem (8.65)]{G} we have that
there exist constants $k_2,\d_2$ such that for any $x\in\bbZ^d$ and for any $\g>0$:
 \be\la{subexp}
\bbP(\g\,n^d\leq |C_x| < \infty )\leq k_2\,\nep{-\d_2\,n^{d-1}}\,.
\end{equation}
Now we can combine (\ref{pisz}) and (\ref{subexp}) to prove the desired estimate.
For any $n$ we write
$$
\bbP(B(n)\cap \cC_\infty=\emptyset) \leq \bbP(B(n)\cap \cC_\infty=\emptyset\,;\;K_n(\g))
+ \bbP(K_n(\g)^c)\,.
$$
The last term in this expression is bounded using (\ref{pisz}). The first term is bounded by
$$
\bbP(\exists x\in B(n)\,:\;\g\,n^d\leq |C_x| < \infty )\leq \sum_{x\in B(n)}
\bbP(\g\,n^d\leq |C_x| < \infty )\,.
$$
Using (\ref{subexp}) we arrive at $\bbP(B(n)\cap \cC_\infty=\emptyset) \leq k\,\nep{-\d\,n^{d-1}}$
for suitable constants $k,\d>0$. \qed

\medskip

We are now ready to finish the proof of Corollary \ref{cor1}.
It is clear from the previous lemmas that in all cases we have both conditions (\ref{trans2}) and
(\ref{rick2}). Moreover it is easily verified that $(\bbZ^d,\varphi_0)$ and $(\bbZ^d,\varphipa)$
have the same type, when $\varphi_0$ is defined by (\ref{fi0}) with $\varphi=\varphipa$.
This ends the proof of Corollary \ref{cor1}.

\subsection{Proof of Corollary \ref{cor100}}\la{prcor100}
It is easily verified that $(\bbZ^d,\varphi_0)$ and $(\bbZ^d,\varphieb)$
have the same type, when $\varphi_0$ is defined by (\ref{fi0}) with $\varphi=\varphieb$.
Therefore the statement about recurrency follows immediately from Theorem \ref{th1}, Part (iii)
and the fact that in all cases (\ref{rick2}) is satisfied (see the previous Subsection).

To prove the second statement we recall that our domination assumption
and Strassen's theorem imply that on a suitable probability space $(\O, \cP)$ one
can define the random field $(\s_1,\s_2)\in \{0,1\}^{\bbZ^d}\times
\{0,1\}^{\bbZ^d}$ such that $\s_1$  has the same law of the infinite
cluster in a  super--critical Bernoulli site--percolation on $\bbZ^d$,
$\s_2$ has the same law of the random field $\s$ defined in (\ref{sil}) and $\s_1
\leq \s_2$, $\cP$--a.s., i.e. $\s_1 (x) \leq \s_2 (x)$ for all $x\in
\bbZ^d$, $\cP$--a.s. To each $\s_1$ we associate the
nearest--neighbor resistor network $\cN_1$ with nodes $\{x \in
\bbZ^d\,:\, \s_1(x)=1\}$ such that nearest--neighbor nodes are
connected by a conductance of value $ c_0=c_0(L)>0$
to be determined below. From the result of \cite{GKZ} we know that $\cN_1$ is transient a.s.

Now,
for each cube  $Q_{Lx,L}$ intersecting $S$ we fix a point $\bar x\in
Q_{Lx,L}\cap S$ (say the one with least lexicographic order).
If we keep all points $\bar x$'s belonging to the infinite cluster in $\s_2$
and neglect all other points of $S$ we obtain a subnetwork $(\wt S,\varphi)$ of $(S,\varphi)$.
If $c_0$ is sufficiently small we have
$\varphi (y ,z) \geq c_0$
for all $y,z\in S$ such that  $y \in Q_{L x_1,L}$, $z\in Q_{L
x_2,L}$, $x_1, x_2 \in \bbZ^d$,  $|x_1-x_2|=1$. Reasoning as in the proof of
Lemma \ref{lemmino} then  immediately implies the a.s.\ transience of $(S,\varphi)$.
Note that this actually works for any $\varphi\in\Phi_d$.

To prove the third statement we observe  that for stationary DPP
$\psi(t)\leq \nep{-\d\,t^d}$, see Lemma \ref{domenica}. Therefore the claim follows from
Theorem \ref{th1}, Part (ii).

\section{Lower bounds on the effective resistance}\label{coraggio}
 Assume that $\bbP$ is dominated by an i.i.d.\ field $\G$ as stated
before Theorem \ref{fegato} and suppose the domination property
holds with some fixed $L\in\bbN$. We shall write $Q_v$ for the cube
$Q_{vL,L}$. To prove Theorem \ref{fegato} and Theorem \ref{combine}
we only need to show that, given $v_0 \in \bbZ^d$, for $\bbP$--a.a.\
$S$ there exists a positive constant $c$ such that for all $x\in S
\cap Q_{v_0 L ,L}$ the lower bounds \eqref{rn1} and \eqref{rn0} on
$R_n(x)$ hold. We restrict to  $v_0=0$, since the general case can
be treated similarly.

 We start by making a first reduction of the network which uses the
stochastic domination assumption. This procedure works in any
dimension $d$.  First, we note that it is sufficient to prove the
bounds in the two theorems for the quantity  $\hat  R_n(x)$, defined
as the effective resistance from $x$ to $Q_{0,2Ln}^c$, instead of
$R_n(x)$ which is the effective resistance from $x$ to $Q_{0,2n}^c$.
In particular, there is no loss of generality in taking $L=1$, in
which case $\hat R_n(x)=R_n(x)$.

The next observation is that, by monotonicity, $R_n(x)$ is
larger than the same quantity computed in the network obtained by
collapsing in a single node $v$ all points in each cube $Q_v$,
$v\in\bbZ^d$.  We now have a network with nodes on the points of
$\bbZ^d$ (although some of them may be empty). Note that across two
nodes $u,v$ we have $N_u N_v$ wires each with a resistance bounded
from below by
$$
\r_{u,v}:=c\,|u-v|^{d+\a}\,
$$ for a suitable (non--random) constant $c>0$.
Moreover, using the stochastic domination assumption we know that
$N_u\leq \G_u$ for all $u\in\bbZ^d$, and we can further lower the
resistance by considering the network where each pair of nodes $u,v$
is connected by $\G_u\G_v$ wires each with the resistance
$\r_{u,v}$. Moreover, we can further lower the resistance by
adding a point to the origin. Hence, from now on,  we
understand that $\G_u$ is replaced by $\G_u+1$ if $u=0$.
We call $(\G,\r)$ this new network. Thus the results will follow
once we prove that for $(\G,\r)$ the effective resistance from $0$
to $Q_{0,2n}^c=\{u\in\bbZ^d:\;\|u\|_\infty > n\}$ satisfies
the desired bounds. From now on we consider the cases $d=1$
and $d=2$ separately.

\subsection{Proof of Theorem \ref{fegato}}
Set $d=1$. We further reduce the network $(\G,\r)$ introduced above
by collapsing in a single node $\wt v$ each pair $\{v,-v\}$. This
gives a network on $\{0,1,2,\dots\}$ where across each pair $0\leq i
< j$ there are now $\wt \G_i\wt \G_j$ wires, where
$\wt\G_i:=\G_i+\G_{-i}$ ($i\neq 0$) and $\wt \G_0:=\G_0$
(recall that by $\G_0$ we now mean the original $\G_0$ plus $1$).
Each of these wires has a resistance at least $\r_{i,j}$ and thus we
further reduce the network by assigning each wire the same
resistance $\r_{i,j}$. We shall call $(\wt \G,\r)$ this new network
and $\wt R_n(0)$ its  effective resistance from $0$ to
$Q_{0,2n}^c$.

An application of the variational formula (\ref{var_con}) to the
network $(\wt\G,\r)$ yields  the upper bound
 \be\la{est1} \wt
R_n(0)^{-1}=\wt C_n(0)\leq \frac1{f_n^2}\sum_{i=0}^n
\sum_{j=i+1}^\infty \wt\G_i\wt\G_j\,(j-i)^{-1-\a} (f_j-f_i)^2 \,,
\end{equation}
for any   sequence  $\{f_i\}_{i\geq 0} $ such that $f_i=f(i)$, $f$
  being a non--decreasing  function   on
  $[0,\infty)$ taking value $0$ only at the origin.

Next,   we choose  $f$ as
\begin{equation}
f(x):= \int_0^xg_\alpha(t)dt, \quad g_\alpha(t) := \left(
  1+\int_0^{t}\left(
    1\wedge\frac{s^2}{s^{1+\alpha}}
  \right)ds
\right)^{-1}\,.
\end{equation}
Note that $f$ satisfies  the differential equation
\begin{equation}
f'(t)^2\left(1+\int_0^{t}\left(1\wedge\frac{s^2}{s^{1+\alpha}}\right)ds\right)
= f'(t)\,. \label{ed}
\end{equation}
Moreover,  $f$ is increasing on $[0,\infty)$, $f(0)=0$ and
$f_k=f(k)$ behaves as
\begin{equation} \la{asymp}
f_k
 \sim
\begin{cases}
  \log k  & \text{ if } \a=1\,,\\
  k^{\a-1} & \text{ if } 1<\a < 2\,,\\
k /\log k &\text { if } \a =2\,,\\
 k & \text{ if  } \a>2\,.
 \end{cases}
 \end{equation}
Here $f_k\sim a_k$ means that there is a constant
$C\geq 1$ such that $C^{-1}\,a_k\leq f_k\leq C\,a_k$, for all $k\geq C$.
Since $g_\a$ is non--increasing we have the concavity bounds
\be\la{conc}
f_j-f_i\leq g_\a(i)\,(j-i)\,,\qquad j\geq i\geq 0\,.
\end{equation}

Let us first prove the theorem for the easier case $\a>2$. We point
out that here we do not need  condition (\ref{expe}) and a finite
first moment condition suffices. Indeed, set
$\xi_i:=\sum_{j>i}(j-i)^{1-\a}\wt \G_j$. This random variables are
identically distributed and have finite first moment since $\a>2$.
Note that $\wt \G_i$ and $\xi_i$ are independent so that
$\bbE[\G_i\,\xi_i]<\infty$. From the ergodic theorem it follows that
there exists a constant $C$ such that $\bbP$--a.s.\
$$
\sum_{i=0}^n\wt \G_i\,\xi_i \leq C\,n \,,
$$
for all $n$ sufficiently large. Due to \eqref{est1},  we conclude
that $\wt C_n(0)\leq n^{-2}\sum_{i=0}^n\wt \G_i\,\xi_i \leq
C\,n^{-1}$ and the desired bound $\wt R_n(0)\geq c\,n$ follows.

The case $1\leq \a\leq 2$ requires more work. Thanks to our choice
of  $f$, we shall prove the following deterministic estimate.
\begin{Le}\la{est10}
There exists a constant $C<\infty$ such that for any $\a\geq 1$
\be\la{est100}  X_i:=\sum_{j=i+1}^\infty \,(j-i)^{-1-\a}
(f_j-f_i)^2 \leq C\,g_\a(i)\,,\quad \; i\in\bbN \,.
\end{equation}
\end{Le}

Let us assume the validity of Lemma \ref{est10} for the moment and
define the random variables
$$
\xi_i:= \sum_{j=i+1}^\infty
\wt\G_j\,(j-i)^{-1-\a}
(f_j-f_i)^2\,.
$$
Let us show that $\bbP$-a.s.\ $\xi_i$ satisfies the same bound as
 $X_i$ in Lemma \ref{est1}. Set
$\L(\l):=\log\bbE[\nep{\l\wt\G_i}]$. From assumption (\ref{expe}) we
know $\L(\l)<\infty$ for all $\l\leq\e$ for some $\e>0$. Moreover,
$\L(\l)$ is convex and $\L(\l)\leq c\,\l$ for some constant $c$, for
all $\l\leq \e$.
  Therefore, using Lemma \ref{est1} we have, for some new constant $C$:
\begin{align}
\bbE[\nep {a_i\xi_i}] &= \prod_{j>i}
\exp{\left[\L(a_i (j-i)^{-1-\a}(f_j-f_i)^2)\right]}\nonumber\\
&\leq \prod_{j>i}\exp {\left[c\,a_i (j-i)^{-1-\a}(f_j-f_i)^2\right]}\nonumber\\
& = \exp {[c\,a_i\,X_i]} \leq \exp {[C\,a_i\, g_\a(i)]}\,,
\la{expo1}
\end{align}
 provided the numbers $a_i>0$ satisfy $a_i(j-i)^{-1-\a}(f_j-f_i)^2\leq \e$ for all $j>i>0$.
Note that the last requirement is satisfied
by the choice $a_i:=\e/g_\a(i)^2$ since, using (\ref{conc}):
$$
a_i(j-i)^{-1-\a}(f_j-f_i)^2\leq a_i(j-i)^{1-\a}g_\a(i)^2\leq a_i\,g_\a(i)^2\,,
$$
for $j>i$, $\a\geq 1$.
This will be our choice of $a_i$ for $1\leq \a\leq 2$. From (\ref{expo1}) we have
\begin{align}\la{proest}
\bbP(\xi_i>2c_1\,\e^{-1}\,g_\a(i)) &\leq \exp{(-2a_ic_1\e^{-1}g_\a(i))}
\bbE[\nep{a_i\xi_i}]
\\& \leq \exp{(-2c_1\,\e^{-1}\,a_i\,g_\a(i))}\exp{(C\,a_i \,g_\a(i))}
\leq \nep{-c_1\,\e\,g_\a(i)^{-1}}\,,
\nonumber
\end{align}
if $c_1$ is large enough. Clearly, $g_\a(i)\leq C(\log i)^{-1}$ for
$1\leq \a\leq 2$ and $i$ large enough. Therefore, if $c_1$ is
sufficiently large, the left hand side in (\ref{proest}) is summable
in $i\in\bbN$ and the Borel Cantelli lemma implies that
$\bbP$--a.s.\ we have $\xi_i \leq c_2\,g_\a(i)$, $c_2:=2c_1\e^{-1}$,
for all $i\geq i_0$, where $i_0$ is an a.s.\ finite random
number.

Next, we write
\be\la{fio}
\sum_{i=0}^n
\sum_{j=i+1}^\infty
\wt\G_i\wt\G_j\,(j-i)^{-1-\a}
(f_j-f_i)^2 \leq \sum_{i=0}^{i_0}\wt\G_i\xi_i + c_2\sum_{i=1}^n\wt\G_i\,g_\a(i)\,.
\end{equation}
The first term is  an a.s.\ finite random number. The second term is
estimated as follows. First, note that \be\la{sumg}
\sum_{i=1}^n\,g_\a(i)\leq f_n\,,
\end{equation} since by concavity
$$
f_{n}=\sum_{j=0}^{n-1} (f_{j+1}-f_j) \geq \sum_{j=0}^{n-1} g_\a(j+1) = \sum_{j=1}^n g_\a(i)\,.
$$
Then we estimate
$$
\bbP(\sum_{i=1}^n\wt\G_i\,g_\a(i)>2c_3\,f_n)
\leq \nep{-2c_3\,f_n}\prod_{i=1}^n\bbE[\nep{\wt\G_i\,g_\a(i)}]
= \nep{-2c_3\,f_n}\,\nep{\sum_{i=1}^n\L(g_\a(i))}\,.
$$
If $i$ is large enough (so that $g_\a(i)\leq \e$) we can estimate
$\L(g_\a(i))\leq c\,g_\a(i)$.
Using (\ref{sumg}) we then have, for $c_3$ large enough
$$
\bbP(\sum_{i=1}^n\wt\G_i\,g_\a(i)>2c_3\,f_n)
\leq \nep{-c_3\,f_n}\,.
$$
Since $f_n\geq \log n$ for all $\a\geq 1$ we see that, if $c_3$ is
sufficiently large, the Borel Cantelli lemma implies that the second
term in (\ref{fio}) is $\bbP$--a.s.\ bounded by $2c_3\,f_n$ for all
$n\geq n_0$ for some a.s.\ finite random number $n_0$. It follows
that there exists an a.s.\ positive constant $c>0$ such that
$\wt R_n(0) \geq c\,f_n$. The proof of Theorem \ref{fegato} is
thus complete once we prove the deterministic estimate in Lemma
\ref{est10}.

\medskip

\noindent {\em Proof of Lemma \ref{est10}}. We only need to consider
the cases $\a\in[1,2]$.  We divide the sum in two terms
\begin{equation}
X_i = \sum_{j=i+1}^{2i} \frac{(f_j-f_i)^2}{(j-i)^{1+\alpha}}
+\sum_{j>2i} \frac{(f_j-f_i)^2}{(j-i)^{1+\alpha}}\,.
\end{equation}
We can estimate the first term
by using the concavity of $f$
and equation (\ref{ed}):
\begin{equation}
\sum_{j=i+1}^{2i} \frac{(f_j-f_i)^2}{(j-i)^{1+\alpha}} \leq
g^2_\alpha(i)\sum_{k=1}^{i} \frac{k^2}{k^{1+\alpha}} \leq C\,
g_\alpha(i)\,.
\end{equation}
As far as the second term is concerned, first observe that
${(f_j-f_i)}/{(j-i)}$ is non-increasing in $j$ (by concavity of $f$)
and so is the general term of the series. As a consequence
\begin{equation}
\sum_{j>2i} \frac{(f_j-f_i)^2}{(j-i)^{1+\alpha}} \leq
\int_{2i}^{+\infty}\frac{(f(x)-f(i))^2}{(x-i)^{1+\alpha}}dx\,.
\end{equation}
In the case $\alpha = 1$ we get, for any $i\geq 1$,
\begin{multline}
\sum_{j>2i} \frac{(f_j-f_i)^2}{(j-i)^{1+\alpha}} \leq
\int_{2i}^{+\infty}\left(\frac{1}{x-i}\ln\frac{1+x}{1+i}\right)^2dx
\leq \\ \frac{1}{i} \int_{2i}^{+\infty}\left(\frac{1}{\frac{x}{i}-1}
\ln\frac{x}{i}\right)^2 \frac{dx}{i} \leq 2g_\alpha(i)
\int_{2}^{+\infty}\left(\frac{\ln t}{t-1}\right)^2dt\,.
\end{multline}
In the case $\alpha>1$ we have,
\begin{equation}
\sum_{j>2i} \frac{(f_j-f_i)^2}{(j-i)^{1+\alpha}} \leq
{2}^{1+\alpha}\int_{2i}^{+\infty}\frac{f^2(x)} {x^{1+\alpha}}dx\,,
\end{equation}
so that, for  $1<\alpha<2$, there are two positive constants $C$ and
$C'$ such that
\begin{equation}
 \sum_{j>2i} \frac{(f_j-f_i)^2}{(j-i)^{1+\alpha}} \leq
{2}^{1+\alpha}
\int_{2i}^{+\infty}C^2\frac{x^{2\alpha-2}}{x^{1+\alpha}}dx \leq
\frac{{2}^{1+\alpha}C^2}{2-\alpha}(2i)^{\alpha-2} \leq C'g_\alpha(i)
\end{equation}
and, for $\alpha=2$, there are two positive
constants $C$ and $C'$
such that
\begin{equation}
\sum_{j>2i} \frac{(f_j-f_i)^2}{(j-i)^{1+\alpha}} \leq
8\int_{2i}^{+\infty}\frac{C^2}{x\log^2x}dx \leq \frac{8C^2}{\log 2i}
\leq C'g_\alpha(i)\,.
\end{equation}
\qed

\subsection{Proof of Theorem \ref{combine}}\label{bambi}
To prove Theorem \ref{combine} we shall make a
series of network reductions which allow us to arrive at a nearest neighbor
one--dimensional problem. We start from the network $(\G,\r)$ defined at the beginning of this
section.

We write $F_a= \{u\in\bbZ^2:\;   \|u\|_\infty=a\}$, $a\in\bbN$. The next reduction is
obtained by collapsing all nodes $u\in F_a$ into a single node for each $a\in\bbN$.

Once all nodes in each $F_a$ are identified we are left with a
one-dimensional network with nodes $a\in\{0,1,\dots\}$. Between
nodes $a$ and $b$ we have a total of $\sum_{u\in F_a}\G_u\sum_{v\in
F_b}\G_v$ wires, with a wire of resistance $\r_{u,v}$ for each $u\in
F_a$ and $v\in F_b$. Finally, we perform a last reduction which
brings us to a nearest--neighbor one--dimensional network.  To this
end we consider a single wire with resistance $\r_{u,v}$ between
node $a$ and node $b$, with $a<b-1$. This wire is equivalent to a
series of $(b-a)$ wires, each with resistance $\r_{u,v}/(b-a)$. That
is we can add $(b-a-1)$ fictitious points to our network in such a
way that the effective resistance does not change. Moreover the
effective resistance decreases if each added point in the series is
attached to its corresponding node $a+i$, $i=1,\dots,b-a-1$, in the
network. If we repeat this procedure for each wire across every pair
of nodes $a<b-1$ then we obtain a nearest neighbor network where
there are infinitely many wires in parallel across any two
consecutive nodes. In this new network, across the pair $i-1,i$ we
have a resistance $R_{i-1,i}$ such that \be\la{rii}
\phi_i:=R_{i-1,i}^{-1} = \sum_{a<i}\sum_{b\geq i}\sum_{u\in
F_a}\sum_{v\in F_b} (b-a)\,\G_u\G_v\,\r_{u,v}^{-1}  \,.
\end{equation}
Moreover, the reductions described above show that
$$ R_n(x)\geq \sum_{i=1}^{n+1}R_{i-1,i} \,.$$
Therefore Theorem \ref{combine} now follows from the estimates
on $R_{i-1,i}$ given in the next lemma.
\begin{Le}\label{ciccolo}
There exists a positive constant $c$ such that
$\bbP$--almost surely, for $i$ sufficiently large
\begin{equation}\label{squadra}
  R_{i,i+1} \geq\,c\,
 \begin{cases}
 i^{-1} & \text{ if } \;\a>2\,,\\
(i \log i )^{-1} & \text{ if } \; \a=2\,.
\end{cases}
 \end{equation}
\end{Le}

\proof
We  first
show that $\bbE( \phi_i )\leq C  \o_i$,
 where $ \o_i =i$ if $\a>2$ and $\o_i= i \log i$ if $\a=2$, where
$\bbE$ denotes expectation w.r.t.\ the field $\{\G_u,\;u\in\bbZ^2\}$.

Thanks to Lemma \ref{bimbo} given in the Appendix, from
(\ref{rii}) we have
\begin{equation}\la{split1}
\bbE(\phi_i) \leq c_1\,\sum_{a<i}\sum_{b\geq i}a(b-a)^{-\a}\,.
\end{equation}
Next we estimate 
$\sum_{b\geq i}(b-a)^{-\a}\leq c_2\,(i-a)^{1-\a}$, so that
using the Riemann integral we obtain
\begin{align*}
\bbE(\phi_i) &\leq c_2\sum_{a<i}a(i-a)^{1-\a}
=c_2 i^{2-\a}
\sum_{a<i}
\frac{a}{i}\left(1-\frac{a}{i}\right)^{1-\a}
\\ &\leq c_3 i^{3-\a}
\int _0 ^{1-\frac{1}{i}} y(1-y)^{1-\a}dy \leq c_3 i^{3-\a} \int _{1/i} ^1 y^{1-\a}dy
\leq c_4\, \o_i\,.
\end{align*}

Hence, for $C$ large we can estimate
\begin{equation}
\bbP ( \phi _i \geq  2C  \o_i ) \leq \bbP (
\phi _i - \bbE (\phi_i) \geq  C \o_i ) \leq (C\,\o_i
)^{-4} \bbE \left[ \bigl( \phi _i -\bbE  (\phi_i)
\bigr)^4 \right]\,,
\end{equation}
where we use $\bbP$ to denote the law of the variables $\{\G_u\}$.

The proof then follows from the Borel--Cantelli Lemma and the following estimate to be established
below: There exists $C<\infty$ such that for all $ i \in \bbN$
\begin{equation}\label{incantesimo}
\bbE  \left[ \bigl( \phi _i -\bbE(\phi_i)
\bigr)^4 \right] \leq C\,i^2\,.
\end{equation}
To prove (\ref{incantesimo}) we write
\be\label{trastevere}
\bbE
\left[ \bigl( \phi _i -\bbE (\phi_i)
\bigr)^4 \right] = \sum_{{\bf a}}\sum_{{\bf b}}\sum_{{\bf u}\sim {\bf a}}\sum_{{\bf v}\sim {\bf b}}\Phi({\bf u},{\bf v})\,G({\bf u},{\bf v})\,,
\end{equation}
where the sums are over ${\bf a}=(a_1,\dots,a_4)$, ${\bf b}=(b_1,\dots,b_4)$
such that $a_k<i\leq b_k$,
${\bf u}\sim {\bf a}$ stands for the set of ${\bf u}=(u_1,\dots,u_4)$ such that
$u_k\in F_{a_k}$, and we have defined, for ${\bf u}\sim {\bf a}$, ${\bf v}\sim {\bf b}$:
$$
\Phi({\bf u},{\bf v})=\prod_{k=1}^4(b_k-a_k)
\r_{u_k,v_k}\,,\quad G({\bf u},{\bf v})=\prod_{k=1}^4
\left(\G_{u_k}\G_{v_k}-\bbE[\G_{u_k}\G_{v_k}]\right)\,.
$$
From the independence assumption on the field $\{\G_u\}$ we know that
$G({\bf u},{\bf v})=0$ unless for every $k=1,\dots,4$ there exists a
$k'=1,\dots,4$ with $k\neq k'$ and $\{u_k,v_k\} \cap
\{u_{k'},v_{k'}\} \neq \emptyset$.
Moreover, when this condition is satisfied using (\ref{fourth})
we can easily bound
$G({\bf u},{\bf v})\leq C$ for some constant $C$.

By symmetry we may then estimate
\begin{align}
&\sum_{{\bf a}}\sum_{{\bf b}}\sum_{{\bf u}\sim {\bf a}}\sum_{{\bf v}\sim {\bf b}}\Phi({\bf u},{\bf v})\,G({\bf u},{\bf v})\nonumber\\
&\leq C\sum_{{\bf a}}\sum_{{\bf b}}\sum_{{\bf u}\sim {\bf a}}\sum_{{\bf v}\sim {\bf b}}\Phi({\bf u},{\bf v})\,\chi\left(
\forall k\,\exists k'\neq k:\; \{u_k,v_k\} \cap
\{u_{k'},v_{k'}\} \neq \emptyset\,\right)\nonumber\\
&\leq 3\,C\sum_{{\bf a}}\sum_{{\bf b}}\sum_{{\bf u}\sim {\bf a}}
\sum_{{\bf v}\sim {\bf b}}\Phi({\bf u},{\bf v})\,\Big[
\chi\left(u_1=u_2\,;\;u_3=u_4\right)
+ \nonumber\\
&\qquad \qquad\qquad + \chi\left(u_1=u_2\,;\;v_3=v_4\right)
+\chi\left(v_1=v_2\,;\;v_3=v_4\right)\Big]\,.
\la{terms}
\end{align}
We claim that each of the three terms in the summation above is of order $i^2$ as $i$ grows.
This will prove the desired estimate  (\ref{incantesimo}).

The first term in (\ref{terms}) satisfies
\be\la{a1}
\sum_{{\bf a}}\sum_{{\bf b}}\sum_{{\bf u}\sim {\bf a}}
\sum_{{\bf v}\sim {\bf b}}\Phi({\bf u},{\bf v})\,
\chi\left(u_1=u_2\,;\;u_3=u_4\right)\leq A(i)^2\,,
\end{equation}
where
$$
A(i):=\sum_{a_1<i}\sum_{b_1\geq i}\sum_{b_2\geq i}
(b_1-a_1)(b_2-a_1)\sum_{u_1\in F_{a_1}}\sum_{v_1\in F_{b_1}}\sum_{v_2\in F_{b_2}}
\r_{u_1,v_1}\r_{u_1,v_2}\,.
$$
Similarly the third term in (\ref{terms}) is estimated by $B(i)^2$, with
$$
B(i):=\sum_{a_1<i}\sum_{a_2<i}
\sum_{b_1\geq i}
(b_1-a_1)(b_1-a_2)\sum_{u_1\in F_{a_1}}\sum_{u_2\in F_{a_2}}
\sum_{v_1\in F_{b_1}}
\r_{u_1,v_1}\r_{u_2,v_1}\,.
$$
Finally, the middle term in (\ref{terms}) is estimated by the product $A(i)B(i)$.
Therefore, to prove (\ref{incantesimo}) it suffices to show that
$A(i)\leq C\,i$ and $B(i)\leq C\,i$.

Using Lemma \ref{bimbo} we see that
$$
\sum_{u_1\in F_{a_1}}\sum_{v_1\in F_{b_1}}\sum_{v_2\in F_{b_2}}
\r_{u_1,v_1}\r_{u_1,v_2} \leq C\, a_1\,(b_1-a_1)^{-1-\a}\,(b_2-a_1)^{-1-\a}\,.
$$
This bound yields
$$
A(i)\leq C\,\sum_{a_1<i} a_1(i-a_1)^{2-2\a}\,,
$$
for some new constant $C$, where we have used the fact that
$$\sum_{b_2\geq i}(b_2-a_1)^{-\a}\leq C(i-a_1)^{1-\a}\,.$$
Using the Riemann integral we obtain
\begin{align*}
\sum_{a_1<i} a_1(i-a_1)^{2-2\a} &\leq C\,i^{4-2\a}\int_0^{1-1/i}
x(1-x)^{2-2\a} dx\\
& \leq C\,i^{4-2\a}\int_{1/i}^1
x^{2-2\a}dx\leq (2\a-3)C\,i\,.
\end{align*}
This proves that $A(i)=O(i)$. Similarly, from
Lemma \ref{bimbo} we see that
$$
\sum_{u_1\in F_{a_1}}\sum_{u_2\in F_{a_2}}\sum_{v_1\in F_{b_1}}
\r_{u_1,v_1}\r_{u_2,v_1} \leq
C\, b_1^{-1}\,a_1\,a_2\,(b_1-a_1)^{-1-\a}\,(b_1-a_2)^{-1-\a}\,.
$$
Therefore
\begin{align*}
B(i) & \leq C
\sum _{b_1\geq i} b_1^{-1}
 \left[ \sum_{a_1<i}   a_1(b_1-a_1)^{-\a} \right]^2 \\
& \leq  C' \,i^2\,
\sum _{b_1\geq i} b_1^{-1} (b_1-i+1)^{2-2\a}
\leq C''\,i\,,
\end{align*}
where we have used the estimate $$
\sum_{a_1<i}   a_1(b_1-a_1)^{-\a}\leq i\sum_{a_1<i}  (b_1-a_1)^{-\a}\leq C\,i\,(b_1-i+1)^{1-\a}\,,
$$
and the fact that for $\a\geq 2$ we have
$$\sum _{b_1\geq i} b_1^{-1} (b_1-i+1)^{2-2\a}\leq i^{-1}\sum_{k=1}^\infty k^{-2}=C/i\,.$$
\qed

%


\smallskip\smallskip

We remark that
a proof of Theorem \ref{combine} could be obtained by application of the variational principle (\ref{var_con}) as in the proof of Theorem \ref{fegato}. To see this one can start
from the network $(\G,\r)$ introduced at the beginning of this section and choose a trial function that is constant in each $F_a$.
Then, for any non--decreasing sequence $(f_0,f_1,\dots)$  such that
$f_0=0$ and $f_k >0$ eventually, one has $R_n(x)\geq A_n(f)$ where
\be\la{ccn0} A_n(f) = f_n^{-2}\sum_{a=0}^n\sum_{b=a+1}^\infty
(f_b-f_a)^2\sum_{u\in F_a}\G_u\sum_{v\in F_b} \G_v\,|v-u|^{-2-\a}\,.
\end{equation}
We then choose $f_k=\log (1+k)$ for $\a>2$ and $f_k=\log (\log (e + k))$ for $\a=2$ and the desired conclusions will follow from suitable control of the fluctuations of the random sum appearing in
(\ref{ccn0}). Here the analysis is slightly more involved than that in the proof of Theorem
\ref{fegato} and it requires estimates as in (\ref{terms}) above. Moreover, one needs a fifth
moment assumption with this approach instead of the fourth moment condition (\ref{fourth}).
Under this assumption, and using Lemma \ref{bimbo},
it is possible to show that a.s.\ there exists a constant $c$ such that
\be\la{estra}
\sum_{u\in F_a}\G_u\sum_{v\in F_b}
\G_v\,|v-u|^{-2-\a}\leq c\,a\,(b-a)^{-1-\a}\,,\quad\,a<b\,.
\end{equation}
Once this estimate is available the proof follows from simple calculations.

\section{Proof of Proposition \ref{treno1} and Theorem \ref{alex}}

\subsection{Proof of Proposition \ref{treno1}}\label{massimo}

The proof of Proposition \ref{treno1} is based on the following
technical lemma related to renewal theory:

\begin{Le}\label{patrono}
Given $\d>1$, define the probability kernel
\begin{equation}\label{defq}
q_k = c(\d) k^{-\d}   \qquad k \in \bbN\,,
\end{equation}
($c(\d)$ being the normalizing constant $ 1/ \sum _{k\geq 1}
k^{-\d})$ and define recursively the sequence $f(n)$ as
\begin{equation*}
\begin{cases}
f(0)=1\,,\\
f(n)=\sum_{k=0}^{n-1} f(k) q_{n-k}\,,\qquad n\in \bbN\,.
\end{cases}
\end{equation*}
If $1<\d<2$, then
\begin{equation}\label{salina}
\lim _{n\uparrow \infty} n^{2-\d} f(n) = \frac{
\G(2-\d)}{\G(\d-1)}\,.
\end{equation}
\end{Le}
\begin{proof}
Let $\{X_i\}_{i \geq 1}$ be  a family of IID random variables with
$P(X_i=k) = q_k $, $k \in \bbN$. Observe now that $P(X_i \geq k) =
\sum _{s=k}^\infty q_s \sim c\, k^{1-\d}$ since $\d>1$. In
particular, if $1<\d<2$ we can use Theorem B of \cite{Do} and get
(\ref{salina}) with $u(n)$ instead $f(n)$, where $u(n)$ is defined
as follows: Consider the random walk $S_n$ on the set  $\bbN\cup\{0\}$,
starting at $0$, $S_0=0$, and defined as $S_n=X_1+X_2 +\cdots +X_n$
for $n\geq 1$. Given $n \in \bbN$ define $u(n)$ as
$$ u(n):= \bbE \left[ \left|\, \{m \geq 0 \,:\, S_m=n \}\right|\,
\right]= \sum _{m=0}^\infty P(S_m=n)\,.
$$
Trivially $u(0)=1$,  while the Markov property of the random walk
$S_n$ gives for $n\geq 1$ that
$$
u(n)= \sum _{m=1}^\infty \sum_{k=0}^{n-1} P(S_{m-1}=k, \;S_m=n)=
\sum _{m=1}^\infty \sum _{k=0}^{n-1} P(S_{m-1}=k) q_{n-k}= \sum
_{k=0}^{n-1} u(k) q_{n-k}\,.
$$
Hence, $f(n)$ and $u(n)$ satisfy the same system of recursive
identities and coincide for $n=0$, thus implying that $f(n)=u(n)$
for each $n \in \bbN$.
 \end{proof}
We have now all the tools in order to prove Proposition
\ref{treno1}:

\smallskip
\noindent {\sl Proof of Proposition \ref{treno1}}.
We shall exhibit a finite energy  unit flux $f(\cdot,\cdot)$  from $x_0$ to infinity in the
network $(S,\varphi)$.
 To this end we define $f(\cdot,\cdot)$ as follows
\begin{equation}\label{flussimetria}
f(x_i,x_k)= \begin{cases} f(i) q _{k-i} & \text{ if } 0 \leq i <
k \,, \\
-f(x_k,x_i)  & \text{ if } 0\leq k <i \,,\\
0 & \text{ otherwise}\,,
\end{cases}
\end{equation}
where $f(m),q_m$ are defined as in the previous lemma for some $\d
\in (1,2)$ that will be fixed below.

Since $\varphi \geq C \varphipa$,  the energy $\cE(f)$ dissipated by
the flux $f(\cdot,\cdot)$ is
\begin{equation}
 \cE (f)  = \sum _{n=0}^\infty \sum _{k=n+1}^\infty
\frac{f(x_n,x_k)^2}{\varphi (|x_n-x_k|)} \leq c
 \sum _{n=0}^\infty \sum _{k=n+1}^\infty
r_{p,\a}(x_k,x_n)  (f_n q _{k-n} )^2\,,
\end{equation}
where $r_{p,\a}(x,y):=1/\varphipa(|x-y|)$.
Hence, due to the previous lemma we obtain that
$$ \cE(f) \leq c \sum _{n=0}^\infty \sum _{k=n+1}^\infty
r_{p,\a} (x_k,x_n) (1+n) ^{2\d-4} (k-n)^{-2\d}\,.$$ In order to
prove that the energy $\cE(f)$ is finite $\bbP$--a.s., it is enough
to show that $ \bbE ( \cE(f))$ is finite  for some $\d \in (1,2)$.
To this end we observe that, due to assumption (\ref{momenti}) and
since $r_{p,\a}(x_k,x_n)=1\lor (x_n-x_k)^{1+\a}$, it holds
\begin{multline}\label{seriale}\qquad\qquad \bbE( \cE(f)) \leq c_1 \sum _{n=0}^\infty
(1+n)^{2\d-4} \sum _{u=1}^\infty \left[1+\bbE \bigl (\,
|x_u-x_0|^{1+\a}\, \bigr)\right] u ^{-2\d}\\ \leq c_2 \Bigl( \sum
_{n=1}^\infty (1+n)^{2\d-4} \Bigr) \Bigl (\sum _{u=1}^\infty
u^{1+\a-2\d} \Bigr)\,,\qquad
\end{multline}
for suitable constants $c_1,c_2$.
Hence, the mean energy is finite if $2\d-4<-1$ and $1+\a-2\d<-1$. In
particular for each $\a\in (0,1)$ one can fix $\d \in(1,2)$
satisfying the above conditions. This concludes the proof of the
transience of $(S,\varphi)$ for $\bbP$--a.a.\ $S$. It remains to
verify assumption (\ref{momenti}) whenever $\bbP$ is a renewal point
process such that $\bbE((x_1-x_0)^{1+\a})<\infty$. To this end we
observe that by convexity
$$
(x_u-x_0)^{1+\a}=u^{1+\a}\Bigl( \frac{1}{u} \sum_{k=0}^{u-1}
(x_{k+1}-x_k) \Bigr)^{1+\a} \leq u^{1+\a}\Bigl( \frac{1}{u}
\sum_{k=0}^{u-1} (x_{k+1}-x_k)^{1+\a} \Bigr)\,.
$$
Since by the renewal property $(x_{k+1}-x_k )_{k\geq 0}$ is a
sequence of  i.i.d. random variables, the mean of the last
expression equals $ u^{1+\a} \bbE ( (x_1-x_0)^{1+\a})=c u^{1+\a}$.
Therefore, (\ref{momenti}) is satisfied.
 \qed

\subsection{Proof of Theorem \ref{alex}}\label{xela}
We recall that $S_*= \cup_{n\geq 0} C_n$, where
\begin{equation}\label{ananas}
C_n:= \left\{ n e^{k\frac{2i\pi}{n+1}} \in {\mathbb C} :\: k\in \{0,
\dots, n\} \right\}
\end{equation}
and $\bbC$ is identified with $\bbR^2$.
In order  to introduce  more symmetries we  consider a family of
``rotations'' of the $C_n$'s: given $\theta = (\theta_n)_{n\geq 0}$
a sequence of independent random variables with uniform law on
$(-\frac{\pi}{n+1}, +\frac{\pi}{n+1})$ we define
\begin{equation}
C_n^\theta := e^{i\theta_n}C_n
\end{equation}
and for $x$ in $C_n$ we  use the notation
\begin{equation}
x^\theta := e^{i\theta_n}x\in C_n^\theta
\end{equation}
We will construct a unit flow $f^\theta$ from $0$ to infinity on
$S_*^\theta :=\cup_n C^\theta_n$ and will make an average over
$\theta$ to build a new flow $f$ on $S$.
In order to describe the flow
$f^\theta$, we consider the probability kernel $q_k= c(\d) k^{-\d}$,
$\d\in (1,2)$, introduced in Lemma \ref{patrono}. The value of $\d$
will be chosen at the end.  We build $f^\theta$ driving a fraction
$q_{n-m}$ of the total flow arriving in a site $x^\theta \in
C^\theta_m$ to each $C^\theta_n$ with $n>m$,   in such a way that
for each site $y \in C^\theta_n$ the flow received from $x^\theta$
is proportional to $\varphipa(x^\theta, y^\theta)$.  We have then,
for all $n>m$, $x \in C_m$ and $y \in C_n$
\begin{equation}
f^\theta(x^\theta,y^\theta) = q_{n-m} \frac{\varphipa(x^\theta,
y^\theta)}{Z^\theta_n(x^\theta)} f^\theta(x^\theta) \label{ftxy}
\end{equation}
with
\begin{equation}
Z^\theta_n(x^\theta):= \sum_{y \in C_n} \varphipa(x^\theta,
y^\theta)
\end{equation}
and $f^\theta (\cdot)$ defined recursively as
\begin{equation}
f^\theta(y^\theta)  = \begin{cases} 1 & \text{ if } y=0\,,\\
\sum_{m<n}{\sum}_{x \in C_m} q_{n-m} \frac{\varphipa(x^\theta,
y^\theta)}{Z^\theta_n(x^\theta)} f^\theta(x^\theta) & \text{ if }
y\in C_n, \;n >0\,.
\end{cases}
\end{equation}
Note that the quantity
\begin{equation}
f_n := \sum_{y\in C_n} f^\theta(y^\theta) \label{fn}
\end{equation}
is independent from $\theta$ and it is defined recursively by
$$\begin{cases}
f_0 =1 \\
 f_n =
\sum_{m<n} q_{n-m} f_m \,, \; n>0\,.
\end{cases}$$
By Lemma \ref{patrono} and the condition $\d\in (1,2)$
we have
\begin{equation}
f_n \sim c\,n^{\delta-2}\,,\qquad n \geq 1\,. \label{asymfn}
\end{equation}

We can now define our flow $f$ on $(S_*, \varphipa)$. For all $m<n$,
$x\in C_m$ and $y\in C_n$ we set
\begin{equation}
f(x,y) := {\mathbb E}\left[
  f^\theta (x^\theta,y^\theta)
\right]\,,
\end{equation}
where the expectation is w.r.t. $\theta$. Taking the conditional
expectation in~(\ref{ftxy}) we get
\begin{equation}
{\mathbb E}\left[
  f^\theta (x^\theta,y^\theta)
  \Big| \theta_m, \theta_n
\right] = q_{n-m} \frac{\varphipa(x^\theta,
y^\theta)}{Z^\theta_n(x^\theta)} {\mathbb E}\left[
  f^\theta(x^\theta)
  \Big| \theta_m
\right] \,.\label{ce}
\end{equation}
By radial symmetry the last factor does not depend neither on $x$
nor on $\theta_m$ and taking the conditional expectation
in~(\ref{fn}) we get
\begin{equation}
{\mathbb E}\left[
  f^\theta(x^\theta)
  \Big| \theta_m
\right] = \frac{f_m}{m+1}\,.
\end{equation}
Taking the expectation in~(\ref{ce}) we obtain 
\begin{equation}
f(x,y)= q_{n-m} {\mathbb E}\left[
  \frac{\varphi(x^\theta, y^\theta)}{Z^\theta_n(x^\theta)}
\right] \frac{f_m}{m+1}\,.
\end{equation}
By means of  this formula it is simple to estimate the energy $\cE(f)$
dissipated by the flux $f(\cdot,\cdot)$ in the network. Indeed, we
can write
\begin{eqnarray*}
\cE (f) &=& {\displaystyle \sum_{m<n} \sum_{x \in C_m} \sum_{y \in
C_n}
\frac{f^2(x,y)}{\varphipa(x,y)}}\\
&=& {\displaystyle \sum_{m<n} \sum_{x \in C_m} \sum_{y \in C_n}
\frac{q_{n-m}^2}{\varphipa(x,y)}  {\mathbb E}\left[
  \frac{\varphipa(x^\theta, y^\theta)}{Z^\theta_n(x^\theta)}
\right]^2 \frac{f_m^2}{(m+1)^2}}\,.
\end{eqnarray*}
Now we observe that
\begin{equation}
|x-x^\theta|\leq   \pi\,, \qquad \forall x \in C_n
\end{equation}
thus implying that  one can find $a>1$ such that, for all $x\not= y$
in $S_*$,
\begin{equation}
a^{-1} \varphipa(x,y) \leq \varphipa(x^\theta, y^\theta) \leq a
\varphipa(x,y)
\end{equation}
As a consequence, setting
\begin{equation}
Z_n(x):= \sum_{y \in C_n} \varphipa(x, y)
\end{equation}
we get
\begin{eqnarray}\label{dentino}
\cE(f) &\leq& {\displaystyle \sum_{m<n} \sum_{x \in C_m} \sum_{y \in
C_n} \frac{a^4 q^2_{n-m}}{\varphipa(x,y)}  \frac{\varphipa^2(x,
y)}{Z^2_n(x)}
\frac{f_m^2}{(m+1)^2}}\\
&=& {\displaystyle \sum_{m<n} \sum_{x \in C_m} \frac{a^4
q^2_{n-m}}{Z_n(x)}
\frac{f_m^2}{(m+1)^2}}\,.
\end{eqnarray}
By Lemma~\ref{bimbo}, there exists a constant $c>0$ such that for
all  $x\in C_m$ and  $n>m$ it holds
\begin{equation}
Z_n(x) \geq  \frac{c}{(n-m)^{1+\alpha}} \,.\end{equation}
Hence, we
can estimate $\cE(f)$ from above as
\begin{equation}
\cE (f) \leq {c} \sum_{k>0} k^{1+\alpha} q_k^2 \sum_{m \geq
0}\frac{f_m^2}{m+1}
\end{equation}
By~(\ref{asymfn}) this is a finite upper bound when
$$
\left\{
\begin{array}{l}
1 + \alpha - 2\delta < -1 \\
2\delta - 4 -1 < -1
\end{array}
\right. \Leftrightarrow \left\{
\begin{array}{l}
2\delta > 2 + \alpha \\
2\delta < 4
\end{array}
\right.
$$
We can choose $\delta \in ( 1,2)$ to have these relations satisfied
as soon as $\alpha<2$. This implies the transience of $(S_*,
\varphipa)$.

\qed

\appendix

\section{Some deterministic bounds 
}
We consider here the following subsets of $\bbR^2$:
\begin{align*}
& C_n= \left\{ n e^{k\frac{2i\pi}{n+1}} \in {\mathbb C} :\: k\in
\{0, \dots, n\} \right\} \\
&D_n = \left\{
  z \in {\mathbb Z}^2 :\:
  \|z\|_\infty = n
\right\}\,
\end{align*}
where $n\in \bbN $ and the complex plane $\bbC$ is identified with
$\bbR^2$.

\begin{Le}\label{bimbo} Set the  sequence
  $(E_n)_{n\geq 0}$ be equal to   $(C_n)_{n\geq 0}$ or  $(D_n)_{n\geq
  0}$, and
define
\begin{equation}
Z_n(x) =   \sum_{y\in E_n} \frac{1}{|y-x|^{2+\alpha}}
\end{equation}
for any $m,n \in {\mathbb N}$ with $m\neq n$ and for any $x \in
E_m$. Then, there exists a constant $a>1$  depending only on $\a$
such that
\begin{align}
& \frac{a^{-1}}{(n-m)^{1+\alpha}} \leq Z_n(x) \leq
\frac{a}{(n-m)^{1+\alpha}}\,, \qquad\;\;\;\;\;\ \text{ if }\; m<n\,, \label{sirius} \\
& \frac{a^{-1}n}{m(m-n)^{1+\alpha}} \leq Z_n(x) \leq
\frac{an}{m(m-n)^{1+\alpha}} \,, \qquad \text{ if }
\;m>n\,.\label{black}
\end{align}
\end{Le}
\begin{proof} We start with the proof of \eqref{sirius} in the case
$E_n=C_n$. Given  $r>0$, we set $\cC_r=\{z \in \bbR^2:|z|=r\}$.
Since the points of $C_n$ are regularly distributed along the circle
${\cC}_n$ and since their number is asymptotically proportional to
the perimeter of ${\cC}_n$, it is enough to find $a>1$ such that
\begin{equation}
\frac{a^{-1}}{(r-|x|)^{1+\alpha}} \leq I_r(x):= \oint_{{\cC}_r}
\frac{|dz|}{|z-x|^{2+\alpha}} \leq \frac{a}{(r-|x|)^{1+\alpha}}
\label{obji}
\end{equation}
for all $r>0$ and for all $x$ in the open ball $ B(0,r)$ centered at
$0$ with radius $r$. Without loss of generality we can assume $x\in
{\mathbb R}_+$. Then, by the change of variable $z\rightarrow
z/(r-x)$, we obtain that
\begin{equation*}
I_r(x)= \frac{I_s (s-1) }{(r-x)^{1+\a} } \,,\qquad s:=r/(r-x)>1 \,.
\end{equation*}
In order to conclude, we only need to show that there exists a
positive constant $c$ such that  $c^{-1}\leq I_s (s-1) \leq c$ for
all $s>1$. We observe that
\begin{multline}
g(s):=I_s(s-1) = \int_{-\pi}^\pi \frac{
  s d\theta
}{
  \left[
    s^2 +(s-1)^2 - 2 s(s-1)\cos\theta
  \right]^{1+\frac{\alpha}{2}}
}= \\
\int_{-\pi s}^{\pi s } \frac{
   dy
}{
  \left[1 + 2s (s-1) (1-\cos \frac{y}{s})
  \right]^{1+\frac{\alpha}{2}}
}\,. \label{gint}
\end{multline}
The last equality follows from the change of variable $\theta
\rightarrow y:=s \theta$. Since $g$ is a continuous positive
function converging to $2\pi$ as $s\downarrow 1$,  all we have to do
is proving that
\begin{equation}
0<\liminf_{s\rightarrow +\infty}g(s) \leq \limsup_{s\rightarrow
+\infty}g(s) < +\infty \,.\label{obj}
\end{equation}
Since there exists $c>0$ such that $\cos\theta \leq 1 - c \theta^2$
for all $\theta \in [-\pi, \pi]$, whenever  $s\geq 2$  the
denominator in the r.h.s. of \eqref{gint} is bounded from below by $
[1+ c y^2]^{1+\a/2}$. Hence, $g(s)$ is  the integral on the real
line of a function dominated by the integrable function  $ y \mapsto
[1 + c y^2]^{-(1+\a/2)}$ as soon as $s\geq 2$. This allows to apply
the Dominated Convergence Theorem, thus implying that
\begin{equation} \lim_{s\rightarrow
+\infty} g(s) = \int_{-\infty}^{+\infty} \frac{
  dy
}{
  \left[
    1 + y^2
  \right]^{1+\frac{\alpha}{2}}
} \in (0,\infty)\,.
\end{equation}
This shows~(\ref{obj}) and concludes the proof of \eqref{sirius} in
the case $E_n= C_n$.

\medskip

We now prove \eqref{sirius} in the  case $E_n=D_n$. As before it
is enough to find $a>1$ such that, for all $r>0$ and $x\in
B_\infty(0,r):=\{z \in \bbR^2\,:\, \|z\|_\infty <r\}$,
\begin{equation}
\frac{a^{-1}}{(r-\|x\|_\infty)^{1+\alpha}} \leq \tilde J_r(x) \leq
\frac{a}{(r-\|x\|_\infty)^{1+\alpha}} \label{tildeq}
\end{equation}
where
\begin{equation*}
\tilde J_r(x) := \oint_{{\cD}_r} \frac{|dz|}{|z-x|^{2+\alpha}}\,,
\qquad {\cD}_r := \{z\in \bbR^2: \|z\|_\infty=r\}\,.
\end{equation*}
Since all the norms are equivalent, we just have to
prove~(\ref{tildeq}) for some $a>1$ with
\begin{equation*}
J_r(x) := \oint_{{\cD}_r} \frac{|dz|}{\|z-x\|_\infty^{2+\alpha}}
\end{equation*}
instead of $\tilde J_r(x)$. At this point it is possible to compute
explicitly $J_r(x)$ to check~(\ref{tildeq}). We proceed in the
following way. We call ${\cD}_r(x)$ the union of the orthogonal
projections of the square $x+{\cD}_{r-\|x\|_\infty}$ on the four
straight lines that contain the four edges of the square ${\cD}_r$.
In other words ${\cD}_r(x)$ is the set of the points in ${\cD}_r$
that share at least one coordinate with at least one point in
$x+{\cD}_{r-\|x\|_\infty}$. We have
\begin{equation}
J_r(x) = \int_{{\cD}_r \cap {\cD}_r(x)}
\frac{|dz|}{\|z-x\|_\infty^{2+\alpha}} + \int_{{\cD}_r \setminus
{\cD}_r(x)} \frac{|dz|}{\|z-x\|_\infty^{2+\alpha}} \,.\label{eqj}
\end{equation}
Estimating from below the first term in the r.h.s. of~(\ref{eqj}),
we get the lower bound
$J_r(x) \geq 2/(r-\|x\|_\infty)^{1+\alpha}$.
On the other hand~(\ref{eqj}) leads to
\begin{multline}
J_r(x) \leq 4\left(
  \frac{2}{(r-\|x\|_\infty)^{1+\alpha}}
  + 2\int_{r-\|x\|_\infty}^{r-\|x\|_\infty+r}
  \frac{dy}{y^{2+\alpha}}
\right) \leq \\  8\left(
  \frac{1}{(r-\|x\|_\infty)^{1+\alpha}}
  + \int_{r-\|x\|_\infty}^{+\infty}
  \frac{dy}{y^{2+\alpha}}
\right) =
\frac{8\left(1+\frac{1}{1+\alpha}\right)}{(r-\|x\|_\infty)^{1+\alpha}}
\end{multline}
and this concludes the proof of  \eqref{sirius} for $E_n=D_n$.

\medskip

To prove \eqref{black}  we first look at the case $E_n=C_n$. Once
again it is enough to find $a>1$ such that, for all $x\in \bbR^2$
and $r<|x|$,
\begin{equation}
\frac{a^{-1}r}{|x|(|x|-r)^{1+\alpha}} \leq I_r(x) \leq
\frac{ar}{|x|(|x|-r)^{1+\alpha}}\,. \label{objii}
\end{equation}
Since $I_r(x)$ depends on $r$ and $|x|$ only, we have
\begin{equation}
I_r(x) = \frac{1}{2\pi|x|} \oint_{{\cC}_{|x|}} I_r(z) |dz| =
\frac{1}{2\pi|x|} \oint_{{\cC}_{|x|}} \left(
  \oint_{{\cC}_r}\frac{|dy|}{|z-y|^{2+\alpha}}
\right)|dz|\,.
\end{equation}
Integrating first in $z$, using~(\ref{obji}), then integrating in
$y$, we get~(\ref{objii}).

\medskip

Finally, to prove \eqref{black} in the case $E_n=D_n$ it is enough
to find $a>1$ such that, for all $x\in \bbR^2$ and $r<\|x\|_\infty$,
\begin{equation}
\frac{a^{-1}r}{\|x\|_\infty(\|x\|_\infty-r)^{1+\alpha}} \leq J_r(x)
\leq \frac{ar}{\|x\|_\infty(\|x\|_\infty-r)^{1+\alpha}}\,.
\label{objjj}
\end{equation}
 As before, we define ${\cD}_r(x)$
as
 the union of the orthogonal
projections of the square $x+{\cD}_{r-\|x\|_\infty}$ on the four
straight lines that contain the four edges of the square ${\cD}_r$.
Note that ${\cD}_r(x)$ is not anymore a subset of ${\cD}_r$ but
we still have
\begin{equation}
J_r(x) = \int_{{\cD}_r \cap {\cD}_r(x)}
\frac{|dz|}{\|z-x\|_\infty^{2+\alpha}} + \int_{{\cD}_r \setminus
{\cD}_r(x)} \frac{|dz|}{\|z-x\|_\infty^{2+\alpha}} \label{eqjj}
\end{equation}
This implies
\begin{eqnarray*}
J_r(x) &\leq&
  \frac{2\min(r,\|x\|_\infty -r)}{(\|x\|_\infty-r)^{2+\alpha}}
  + 8\int_{\|x\|_\infty-r}^{\|x\|_\infty-r+r}
  \frac{dy}{y^{2+\alpha}}
\\
&=&
  \frac{2\min\left(r,\|x\|_\infty -r\right)}
       {(\|x\|_\infty-r)^{2+\alpha}}
  + \frac{8}{1+\alpha}
  \left[
    \frac{1}{(\|x\|_\infty-r)^{1+\alpha}}
    -
    \frac{1}{\|x\|_\infty^{1+\alpha}}
  \right]
\\
&=&
  \frac{2\min\left(r,\|x\|_\infty -r\right)}
       {(\|x\|_\infty-r)^{2+\alpha}} \nonumber\\
&&
  + \frac{8(1+\alpha)^{-1}}{(\|x\|_\infty-r)^{1+\alpha}}
  \left[
    1-\left(
      1-\frac{r}{\|x\|_\infty}
    \right)^{1+\alpha}
  \right]
\end{eqnarray*}
The convexity of $
y \mapsto (1-y)^{1+\alpha}
$
gives
\begin{equation}
1 - (1+\alpha)\frac{r}{\|x\|_\infty} \leq
\left(1-\frac{r}{\|x\|_\infty}\right)^{1+\alpha} \leq
1-\frac{r}{\|x\|_\infty} \label{cvx}
\end{equation}
and observing that
\begin{equation}
\min\left(r,\|x\|_\infty -r\right) \leq \frac{
  2r (\|x\|_\infty -r)
}{
  \|x\|_\infty
}
\end{equation}
we get
\begin{equation}
J_r(x) \leq \frac{12r}{\|x\|_\infty (\|x\|_\infty-r)^{1+\alpha}}
\end{equation}
As far as the lower bound is concerned we distinguish two cases. If
${\cD}_r(x)$ does not contain any vertex of the square ${\cD}_r$
then we can estimate $J_r(x)$ from below with the the first term in
the right-hand side of~(\ref{eqjj}):
\begin{equation}
J_r(x) \geq \frac{1}{(\|x\|_\infty-r)^{1+\alpha}} \geq
\frac{r}{\|x\|_\infty (\|x\|_\infty-r)^{1+\alpha}}
\end{equation}
If ${\cD}_r(x)$ does contain some vertex of the ${\cD}_r$ we
estimate $J_r(x)$ with the the second term in the right-hand side
of~(\ref{eqjj}). Recalling~(\ref{cvx}):
\begin{equation*}
J_r(x) \geq
  \int_{\|x\|_\infty-r}^{\|x\|_\infty}
  \frac{dy}{y^{2+\alpha}}
=
  \frac{(1+\alpha)^{-1}}{(\|x\|_\infty-r)^{1+\alpha}}
  \Big[
    1-\Big(
      1-\frac{r}{\|x\|_\infty}
    \Big)^{1+\alpha}
  \Big]
\geq
  \frac{(1+\alpha)^{-1}r}{\|x\|_\infty(\|x\|_\infty-r)^{1+\alpha}}\,.
\end{equation*}

\end{proof}

\section{The random walk $(\bbZ^d, \varphipa)$}\label{losangeles}

In this Appendix, we study by harmonic analysis the random walk on
$\bbZ^d$ with polynomially decaying jump rates. Without loss of
generality,  we slightly modify the function $\varphipa$ as
$\varphipa (r)= (1+r^{d+\a})^{-1} $. Hence, we consider jump
probabilities \be\la{defc} p(x,y) = p(y-x)\,,\quad \; p(x) =
c\,(1+|x|^{d+\a})^{-1}\,,
\end{equation}
for $c>0$ such that $\sum_{x\in\bbZ^d}p(x)=1$. The associated
homogeneous random walk on $\bbZ^d$ is denoted
$X=\{X_k,\,k\in\bbN\}$.

\subsection{Recurrence and transience}
It is known that $X$ is transient if $d\geq 3$ for any $\a>0$ and in
$d=1,2$ it is transient if and only if $0<\a<\min\{2,d\}$. Let us
briefly recall how this can be derived by simple harmonic analysis.

From \cite{Spitzer}[ Section 8, T1] the random walk is transient if
$d\geq 3$ for any $\a>0$, and it is recurrent in $d=1$ for $\a>1$
and in $d=2$ for $\a>2$. Other cases are not covered by this theorem
but one can use the following facts. Define the characteristic
function
$$\phi(\theta) = \sum_{x\in\bbZ^d}p(x)\nep{i\,x\cdot \theta}\,.$$
Note that $\phi(\theta)$ is real and $-1\leq \phi(\theta)\leq 1$.
Moreover, since the kernel $p$ is aperiodic $\phi(\theta)<1$ for all
$\theta\neq 0$. By the integrability criterion given in
\cite{Spitzer}[Section 8, P1],  $X$ is transient if and only if
\begin{equation}\label{spinone}
\lim _{t \uparrow 1} \int _{[-\p,\p)^d } \frac{1}{1-t \phi(\theta) }
d \theta < \infty\,.
\end{equation}

If $\a\in(0,2)$ we have, for any $d\geq 1$:
\be\la{1} \lim_{|\theta|\to 0}\frac{1-\phi(\theta)}{|\theta|^\a} =
\k_{\d,\a}\in(0,\infty)\,.
\end{equation}
The limit (\ref{1}) is proved in \cite{Spitzer}[ Section 8, E2] in
the case $d=1$ but it can be generalized to any $d\geq 1$. Indeed,
writing $\theta=\e\hat\theta$, $|\hat\theta|=1$: \be\la{id}
\frac{1-\phi(\theta)}{|\theta|^\a} =
\e^d\sum_{x\in\bbZ^d}(1+|x|)^{d+\a}p(x)\frac{1-\cos(\e
x\cdot\hat\theta)}{(\e+|\e x|)^{d+\a}}\,,
\end{equation}
and when $\e\to 0$, using $(1+|x|)^{d+\a}p(x)= c$, we have
convergence to the integral
$$
c\,\int_{\bbR^d} f(x)\,dx\,,\quad\;\; f(x): =
\frac{1-\cos(x\cdot\hat\theta)}{|x|^{d+\a}}\,,
$$
where $\hat\theta$ is a unit vector (the integral does not depend on
the choice of $\hat\theta$). This integral is positive and finite
for $\a\in(0,2)$ and (\ref{1}) follows.

Using (\ref{1}) the integrability criterion   \eqref{spinone}
implies that for $d=1$ the RW is transient if and only if
$\a\in(0,1)$, while  for  $d=2$ the RW is transient for any $\a\in
(0,2)$. The only case remaining is $d=2, \a=2$. This apparently is
not covered explicitly in \cite{Spitzer}. However, one can modify
the argument above to obtain that for any $d\geq 1$, $\a=2$:
\be\la{2} \lim_{|\theta|\to
0}\frac{1-\phi(\theta)}{|\theta|^2\log(|\theta|^{-1})} =
\k_{d,2}\in(0,\infty)\,.
\end{equation}
Thus, using again the integrability criterion \eqref{spinone}, we
see that $\a=2,d=2$ is recurrent.

To prove (\ref{2}) one can write, reasoning  as in (\ref{id}): For
any $\d>0$
\begin{align*}\la{id2}
\frac{1-\phi(\theta)}{|\theta|^2\log(|\theta|^{-1})} &=
\frac{c\,\e^d}{\log(\e^{-1})} \sum_{x\in\bbZ^d:\,1\leq |x|\leq
\e^{-1}\d}
\frac{1-\cos(\e x\cdot\hat\theta)}{(\e+|\e x|)^{d+2}} + O\left(1/\log(\e^{-1})  \right)\\
& =  \frac{c}{2\log(\e^{-1})}
\int_{
\e\leq |x|\leq \d} \frac{x_1^2\,dx}{|x|^{d+2}}+ O\left(
1/\log(\e^{-1})\right)\,,
\end{align*}
where $x_1$ is the first coordinate of the vector
$x=(x_1,\dots,x_d)$. The integral appearing in the first term above
is, apart from a constant factor, $\int_{\e}^\d r^{-1}dr =
\log(\e^{-1}) + {\rm const.}$ This proves the claim (\ref{2}).

\subsection{Effective resistance estimates}
Let  $R_n:=R_n(0)$  be the effective resistance associated to the
box $\{x\in\bbZ^d\,,\; \|x\|_\infty \leq n\}$. As already discussed
in the introduction, $\frac1c\,R_n$ (where $c>0$ is the constant in
(\ref{defc})) equals
 the expected number of visits to the origin before visiting the
set $\{x\in\bbZ^d\,,\; \|x\|_\infty >n\}$ for the random walk $X$
with $X_0=0$.  We are going to give upper bounds on $R_n$ in the
recurrent cases $d=1,2$, $\a\geq \min\{d, 2\}$. By comparison with
the simple nearest neighbor random walk we have that (for any $\a$)
$R_n\leq C\log n$ if $d=2$ and $R_n\leq C\,n$ if $d=1$. Due to
Theorems \ref{fegato} and \ref{combine}, this estimate is of the
correct order whenever $p(x)$ has finite second moment ($\a>2$). The
remaining cases are treated as follows.

We claim that for some constant $C$ \be\la{cl} R_n\leq
C\int_{[-\pi,\pi)^d}\frac{d\theta}{n^{-\a} + (1-\phi(\theta))}\,.
\end{equation}
The proof of (\ref{cl}) is given later. Assuming (\ref{cl}), we
obtain the following bounds:
\begin{equation}\la{rnbos}
R_n\leq C\begin{cases}\log n & d=1,\,\a=1\\
n^{\a-1} & d=1,\,\a\in(1,2)\\
n/\sqrt{\log n}
 & d=1,\,\a=2\\
\log\log n & d=2,\,\a=2\,.\end{cases}
\end{equation}
With the only exception of the case $d=1,\a=2$, the above upper
bounds are of the same order  of the lower bounds of Theorem
\ref{fegato} and \ref{combine}.

The above bounds are easily obtained as follows. For $\a\in[1,2)$,
$d=1$, using the bound $1-\phi(\theta)\geq \l|\theta|^\a$, cf.
(\ref{1}), we see that the first two estimates in (\ref{rnbos})
follow
by decomposing the integral in \eqref{cl} in the regions
$|\theta|\leq n^{-1}$, $|\theta|>n^{-1}$ and then using obvious
estimates.

For $\a=2$, we decompose the integral in \eqref{cl} in the regions
$|\theta|\leq \e $, $|\theta|>\e$, $\e:=1/10$. Since
$1-\phi(\theta)$ vanishes only for $\theta=0$, the  integral over
the region $|\theta|>\e$  is of order 1, while we can use  the bound
$1-\phi(\theta)\geq \l|\theta|^2\log(|\theta|^{-1})$ over the region
$|\theta|\leq \e$, cf.\ (\ref{2}). Hence,
 we see that for some $C$ \be\la{clove} R_n\leq C\int_{[-\e,\e)^d}
\frac{d\theta}{(n^{-2} + |\theta|^2\log(|\theta|^{-1}))}\,.
\end{equation}
Then, if  $d=1$
(\ref{clove}) yields
$$
R_n\leq 2C\int_0^{(n\sqrt{\log n})^{-1}} \frac{d\theta}{n^{-2} } +
2C\int_{(n\sqrt{\log n})^{-1}}^\e
\frac{d\theta}{\theta^2\log(\theta^{-1})}
\,.
$$
The first integral gives $2C\,\frac{n}{\sqrt{\log n}}$. With the
change of variables $y=1/\theta$ the second integral becomes
$$\int_{\e^{-1}}^{n\sqrt{\log n}}
\frac{dy}{\log y}\,.
$$
This gives an upper bound $O(n/\sqrt{\log n})$. (Indeed, for
$\e=1/10$ we have that for every $y\geq \e^{-1}$, $(\log y)^{-1}
\leq 2[(\log y)^{-1} -  (\log y)^{-2}] = 2\frac{d}{dy}\frac{y}{\log
y}$, which implies the claim). Therefore $R_n\leq C\,n/\sqrt{\log
n}$ in the case $d=1,\a=2$.

    Reasoning as above, if $d=2$ and $\a=2$ we have, for some
$C$:
$$
R_n\leq C\int_0^\e \frac{\theta\,d\theta}{(n^{-2} +
\theta^2\log(\theta^{-1}))}\,.
$$
We divide the integral as before and obtain
$$
R_n\leq C\,n^2\,\int_0^{(n\sqrt{\log n})^{-1}} \theta\,d\theta +
C\int_{(n\sqrt{\log n})^{-1}}^\e
\frac{d\theta}{\theta\log(\theta^{-1})}\,.
$$
The first integral is small and can be neglected. The second
integral is the same as
$$
\int_{\e^{-1}}^{n\sqrt{\log n}} \frac{dy}{y\log y} \leq C\log\log
n\,.
$$
This proves that $R_n\leq C\log\log n$\,.

\subsection{Proof of claim (\ref{cl})}
To prove (\ref{cl}) we introduce the truncated kernel
$$Q_n(x,y)= \bbP_x(X_1 = y\,; |X_1 - x| \leq c_1\,n) =
\frac{c}{1+|y-x|^{d+\a}}\,1_{\{|y-x|\leq c_1\,n\}}\,,$$ where
$c>0$ is defined in (\ref{defc}) and $c_1>0$
is another constant. Clearly, for all sufficiently large $c_1$
$$
R_n\leq 
c\sum_{k=0}^\infty Q_n^k(0,0)\,,
$$
where
$Q_n^k(0,0)$ is the probability of returning to the origin
after $k$ steps without ever taking a jump of size larger than
$c_1\,n$.

Note that for any $x$
$$
u_n:=\sum_{y\in\bbZ^d} Q_n(x,y) = \bbP_0(|X_1|\leq c_1 \,n) = 1 -
\g_n\,,\quad \g_n: = \sum_{|x|>c_1\,n} p(x) \sim n^{-\a}\,.
$$
Let $\hat Q_n(x,y)$ denote the kernel of the RW on $\bbZ^d$ with
transition $p(x)$ conditioned to take only jumps of size less than
$c_1n$, so that $\hat Q_n(x,y) = u_n^{-1}Q_n(x,y)$. Set
$$
\phi_n(\theta)=\sum_{x\in\bbZ^d} Q_n(0,x)\nep{i\theta\cdot x}\,,
\quad \hat\phi_n(\theta)=\sum_{x\in\bbZ^d} \hat
Q_n(0,x)\nep{i\theta\cdot x}\,.
$$
$\phi_n(\theta)=u_n\hat\phi(\theta)$ is real and $\nep{i\theta\cdot
x}$ can be replaced by $\cos{(\theta\cdot x)}$ in the above
definitions. We can write
$$
\sum_{k=0}^\infty Q_n^k(0,0) =
\sum_{k=0}^\infty\frac1{(2\pi)^d}\int_{[-\pi,\pi)^d}\phi_n(\theta)^k\,d\theta=
\frac1{(2\pi)^d}\int_{[-\pi,\pi)^d}\frac{d\theta}{1-\phi_n(\theta)}\,,
$$
where we use the fact that $|\phi_n(\theta)|\leq u_n<1$ for any $n$.
Moreover
$$
1-\phi_n(\theta)= 1-u_n + u_n(1-\hat\phi_n(\theta)) = \g_n +
u_n(1-\hat\phi_n(\theta))\,.
$$
Therefore it is sufficient to prove that \be\la{top}
u_n(1-\hat\phi_n(\theta))\geq \d\,(1-\phi(\theta))\,,
\end{equation}
for some constant $\d>0$. Let $B(t)$ denote the euclidean ball of
radius $t>0$. Suppose $\theta=y\,n^{-1}\hat\theta$ for some $y>0$
and a unit vector $\hat\theta$.
Then, 
\begin{equation}\la{expre}
\frac{u_n(1-\hat\phi_n(\theta))}{1-\phi(\theta)} =
\frac{\sum_{x\in\bbZ^d\cap B(c_1n)} \frac{1-\cos(y
(x/n)\cdot\hat\theta)}{(n^{-1}+|(x/n)|)^{d+\a}}} {\sum_{x\in\bbZ^d}
\frac{1-\cos(y (x/n)\cdot\hat\theta)}{(n^{-1}+|(x/n)|)^{d+\a}}}
\end{equation}
Reasoning as in the proof of (\ref{1}) and (\ref{2}) we see that for
all $\a\in(0,2]$, the expression (\ref{expre}) is  bounded away from
$0$ for $y\in(0,C]$, for $n$ large enough. Indeed, if $\a\in(0,2)$
we have convergence, as $n\to\infty$ to
$$
\frac{\int_{B(c_1)}
\frac{1-\cos(yx\cdot\hat\theta)}{|x|^{d+\a}}\,dx}{\int_{\bbR^d}
\frac{1-\cos(yx\cdot\hat\theta)}{|x|^{d+\a}}\,dx}\,.
$$
On the other hand, for $\a=2$, from the proof of (\ref{2}) we see
that (\ref{expre}) converges to $1$. Therefore, in all cases
(\ref{top}) holds for any $|\theta|\leq C\,n^{-1}$, for all $n$
sufficiently large.

Next, we consider the case $|\theta|> C\,n^{-1}$. For this range of
$\theta$ we know that
$$
1-\phi(\theta)\geq \l|\theta|^\a \geq \l\,C^\a\,n^{-\a}\,,
$$
for some $\l>0$. Note that this holds also in the case $\a=2$
according to (\ref{2}). From
$$
\phi(\theta)-u_n\hat\phi_n(\theta) = \sum_{x:\,|x|>c_1\,n}
p(x)\,\cos(\theta\cdot x)\,,
$$
we obtain $\phi(\theta)\geq u_n\hat \phi_n(\theta) - \g_n$.
Therefore, for $|\theta|> C\,n^{-1}$
\begin{align*}
u_n(1- \hat\phi_n(\theta)) - \d\,(1-\phi(\theta))& \geq -2\g_n +
(1-\d)(1-\phi(\theta))
\\
&\geq  -2\g_n + (1-\d)\l\,C^\a\,n^{-\a}\,.
\end{align*}
Taking $C$ large enough and using $\g_n=O(n^{-\a})$ shows that
(\ref{top}) holds. This ends the proof of (\ref{cl}).

\bigskip

\noindent  {\bf Acknowledgements}. The authors  kindly thank M.
 Barlow, F. den Hollander and P. Mathieu for useful discussions.  A.
 Gaudilli\`{e}re  acknowledges the financial support of   GREFI--MEFI.

\end{document}